\documentclass[reqno,12pt]{amsart}
\usepackage{amsmath,amssymb,color}
\usepackage{amsthm}
\usepackage{comment}

\newtheorem{theorem}{Theorem}

\newtheorem{proposition}{Proposition}
\newtheorem{remark}{Remark}
\newtheorem{lemma}{Lemma}
\newtheorem{example}{Example}

\def\re{\mathbb{R}}

\def\({\left(}
\def\){\right)}

\def\|{\Vert}
\def\weakto{\rightharpoonup}
\def\bo{\boldsymbol}

\begin{document}
\title[]{Asymptotic behavior of least energy solutions to the Lane-Emden problem on metric graphs}

\author{Kazuki Sato}

\address{
Department of Mathematics, Osaka Metropolitan University \\
3-3-138, Sumiyoshi-ku, Sugimoto-cho, Osaka, Japan \\
}

\email{sf22817a@st.omu.ac.jp}

\begin{abstract}
In this paper, we study the asymptotic behavior of least energy solutions to Lane--Emden problems on compact metric graphs as $p \to \infty$. For the problem with Dirichlet--Kirchhoff boundary conditions, we characterize the limiting variational problem in terms of the Dirichlet--Kirchhoff Green function and show that least energy positive solutions converge, up to a subsequence, to a normalized Green function centered at a maximizer of the diagonal Green function. We also show that their maximum points approach the set of such maximizers. For the problem with Kirchhoff--Neumann boundary conditions, we characterize the limiting variational problem on compact metric graphs without cycles and show that least energy solutions converge, up to a subsequence, to a limiting profile associated with a pair of points realizing the maximal distance on the graph. Moreover, the distance between their maximum and minimum points converges to the maximal distance on the graph.
\end{abstract}


\date{\today}

\maketitle

\section{Introduction}
In this paper, we study the asymptotic behavior, as $p \to \infty$, of least energy solutions to the following  Lane--Emden problem on a compact metric graph  $\bo{G}$:
\begin{equation}
\label{metric graph}
-\Delta u = u^p\quad \mbox{in}\quad \bo{G},
\end{equation}
where $p > 1$ and we assume the Dirichlet boundary condition at the ends of $\bo{G}$.

The large exponent limit for positive solutions to the Lane--Emden problem has been extensively studied in Euclidean domains. In particular, in the two-dimensional case, Ren--Wei \cite{R-W1994, R-W1996} investigated the concentration phenomenon of variational solutions as $p\to\infty$. They showed that least energy solutions exhibit single point concentration and characterized the location of the concentration point. Moreover, their analysis reveals an important role of the Green function and the associated Robin function in the characterization of the peak location. 

On the other hand, the asymptotic behavior of the Lane--Emden problem has also been studied in domains with a one-dimensional structure. Recently, De Marchis--Mazzuoli--Pacella \cite{M-M-P} considered positive one-dimensional solutions in cylinders and investigated their asymptotic behavior both as $p\to 1$ and as $p\to\infty$.

For nonlinear elliptic problems on metric graphs, Kurata--Shibata \cite{Kurata-Shibata} studied least energy positive solutions to
\[
-\varepsilon^2\Delta u+u=u^p
\]
on compact metric graphs as $\varepsilon\to 0$. They showed that, for sufficiently small $\varepsilon$, a least energy solution has a unique local maximum point, around which the solution concentrates, and investigated the location of this maximum point in relation to the edge lengths of the graph. In contrast to their small parameter limit, in the present paper we consider the large exponent limit $p\to\infty$ for the Lane--Emden problem on a fixed compact metric graph.

The aim of this paper is to describe the asymptotic behavior of least energy solutions to the Lane--Emden problem on compact metric graphs with Dirichlet--Kirchhoff conditions as $p\to\infty$, as well as the asymptotic location of their maximum points. To this end, we study the limiting variational problem associated with the optimal Sobolev constants and characterize its minimizers in terms of the Dirichlet--Kirchhoff Green function. We show that the limiting profile of least energy solutions is given by a normalized Green function whose pole maximizes the diagonal function
\[
a \longmapsto G(a,a).
\]
We also show that the maximizers of $G(a,a)$ cannot be vertices of the graph and that the distance between the maximum points of least energy solutions and the set of these maximizers tends to zero as $p\to\infty$.

To state our setting and main results, we first introduce some notation from graph theory.
\begin{itemize}
\item $\bo{G} = \bo{G}(\bo{V}, \bo{E})$ is a graph, where $\bo{V}$ is a set of vertices and $\bo{E}$ is a set of edges. We always assume that $\bo{G}$ is connected and the number of edges $\# \bo{E}$ is finite.
\item $\bo{G}$ is a metric graph if each edge $\bo{e} \in \bo{E}$ is isometric to an interval $[0, l(\bo{e})]$ (or $[0, l(\bo{e}))$~if~$l(\bo{e}) = \infty$), where $l(\bo{e}) \in (0, \infty]$ denotes the length of $\bo{e}$. We identify the edge $\bo{e}$ with $[0, l(\bo{e})]$.
\item A metric graph $\bo{G}$ is said to be compact if $l(\bo{e}) < \infty$ for each $\bo{e} \in \bo{E}$.
\item $\bo{G}$ contains no loops.
\item $\bo{e} \succ \bo{v}$ means that $\bo{e}$ is incident to $\bo{v} \in \bo{V}$.
\item $\deg\bo{v}$ denotes the number of edges that are incident to $\bo{v}$. We assume that $\deg\bo{v}$ $\ne 2$ for any $\bo{v} \in \bo{V}$.
\item $\bo{V}_{int}$ is the set of all vertices with $\deg\bo{v}$ $\geq 3$.
\item $\bo{V}_{end}$ is the set of all vertices with $\deg\bo{v}$ $= 1$. Hence, $\bo{V}_{int}\sqcup \bo{V}_{end} = \bo{V}$.
\item A \textit{cycle} is a closed path in $\bo{G}$ consisting of distinct edges and vertices, except for its initial and terminal vertex, which coincide.
\end{itemize}
\vskip\baselineskip
To study problem \eqref{metric graph}, we employ a variational framework. Let $H^1(\bo{G})$ denote the space of all continuous functions $u$ on $\bo{G}$ such that $u^{(\bo{e})} \in H^1(\bo{e})$ for each edge $\bo{e} \in \bo E$, where $u^{(\bo{e})}$ is the restriction of $u$ to~$\bo{e}$. 
It is straightforward to verify that $H^1(\bo{G})$ is a Hilbert space endowed with norm
\[
\|u\|^2_{H^1(\bo{G})} := \int_{\bo{G}} \left( |\nabla u|^2 + u^2\right)dx := \sum\limits_{\bo{e}\in \bo{E}} \int_0^{l(\bo{e})} \left( |\nabla u^{(\bo{e})}|^2 + (u^{(\bo{e})})^2\right) dx,
\]
where $\nabla = \frac{d}{dx}$. The $L^r(\bo G)$-norm is defined analogously for $r\in[1,\infty]$, and
\[
H_0^1(\bo{G}) := \left\{u \in H^1(\bo{G}) \middle| u(\bo{v}) = 0\ \mbox{for each}\ \bo{v} \in \bo{V}_{end} \right\}.
\]
For each $y \in \bo{G}\setminus \bo{V}_{end}$, we denote by $G(\cdot,y)$ the Dirichlet--Kirchhoff Green function associated with the Laplacian on $\bo{G}$. It is defined as the unique function $G(\cdot,y) \in H_0^1(\bo{G})$ such that
\begin{equation}
\int_{\bo{G}} \nabla_x G(x,y) \nabla \varphi(x) dx = \varphi(y),\qquad \text{for all~~} \varphi \in H_0^1(\bo{G}).
\label{eqn;GF}
\end{equation}
Equivalently, $G$ satisfies
\[
-\Delta_x G(x,y) = \delta_y \quad \mbox{on}\quad \bo{G},
\]
together with the Dirichlet conditions at the boundary vertices and the Kirchhoff conditions at the interior vertices. The Green function is symmetric, that is,
\[
G(x,y)=G(y,x)
\]
for every $x,y\in \bo{G}\setminus\bo{V}_{end}$. Indeed, taking $\varphi=G(\cdot,x)$ and $\varphi=G(\cdot,y)$
in \eqref{eqn;GF}, respectively, we obtain
\[
G(x,y) = \int_{\bo{G}}\nabla_z G(z,x) \nabla_z G(z,y)dz = G(y,x).
\]
We further define the set of maximizers of the diagonal Green function by
\[
M := \left\{ a \in \bo{G} ~\middle|~ G(a, a) = \max_{x \in \bo{G}}G(x,x) \right\}.
\]
Let $J_p$ be a functional on $H^1(\bo{G})$ such that
\[
J_p(u) := \frac{1}{2}\int_{\bo{G}} |\nabla u|^2 dx - \frac{1}{p+1}\int_{\bo{G}} |u|^{p+1} dx.
\]
Then, $J_p \in C^1(H^1(\bo{G}), \re)$. The critical points of the restriction of $J_p$ to $H_0^1(\bo{G})$ satisfy \eqref{E-L} as the Euler--Lagrange equation:
\begin{align}
\label{E-L}
\begin{cases}
-\Delta u_p^{(\bo{e})} = (u_p^{(\bo{e})})^p &\mbox{for each edge}\ \bo{e} \in \bo{E},\\
u_p^{(\bo{e})} > 0 &\mbox{for each edge}\ \bo{e} \in \bo{E},\\
\sum\limits_{\bo{e} \succ \bo{v}}\partial u_p^{(\bo{e})}(\bo{v}) = 0 &\mbox{for each vertex}\ \bo{v} \in \bo{V}_{int},\\
u^{(\bo{e})}(\bo{v}) = 0 &\mbox{for each vertex}\ \bo{v} \in \bo{V}_{end},\\
u_p^{(\bo{e})}(\bo{v}) = u_p^{(\bo{e}')}(\bo{v}) &\mbox{if}\ \bo{e} \succ \bo{v}\ \mbox{and}\ \bo{e}' \succ \bo{v},
\end{cases}
\end{align}
where $\Delta = \frac{d^2}{dx^2}$, and $\partial u_p^{(\bo{e})}(\bo{v})$ is the outward derivative of $u_p^{(\bo{e})}$ at $\bo{v}$. In \eqref{E-L}, the third line represents the Kirchhoff condition, the fourth line the Dirichlet boundary condition, and the last line the continuity condition at $\bo v$. Put
\[
\sigma_p := \inf_{u \in H_0^1(\bo{G})\setminus \{0\}} \sup_{t>0} J_p(tu).
\]
Then, for each $p > 1$, there exists a positive solution $u_p$ such that $J_p(u_p) = \sigma_p$. 
We call $u_p$ a least energy solution. Our results describe the asymptotic profile of the least energy solution as $p \to \infty$.
\begin{theorem}
\label{main}
Assume that $\bo{G}$ is a compact metric graph with $\bo{V}_{int} \ne \emptyset$ and $\bo{V}_{end} \ne \emptyset$. For each $p>1$, let $u_p$ be a least energy solution. Then, 
\[
M \subset \bo{G} \setminus \bo{V}.
\]
Moreover, for every sequence $\{p_n\}$ with $p_n \to \infty$, there exist a subsequence, still denoted by $\{p_n\}$, and a point $a_* \in M$ such that
\[
u_{p_n} \longrightarrow \frac{G(\cdot, a_*)}{G(a_*,a_*)}
\]
strongly in $H_0^1(\bo{G})$ and uniformly on $\bo{G}$.
\end{theorem}
Theorem~\ref{main} describes the asymptotic behavior of least energy solutions in terms of the Dirichlet--Kirchhoff Green function. The following theorem shows that the distance between the maximum points and the set $M$ tends to zero as $p \to \infty$.
\begin{theorem}
\label{main2}
Let $u_p$ be the least energy solution of \eqref{metric graph}, and let $x_p \in \bo{G}$ be a maximum point of $u_p$. Then
\[
dist(x_p, M) \longrightarrow 0\qquad \mbox{as } p \to \infty.
\]
In particular, if $M = \{a_*\}$, then
\[
x_p \longrightarrow a_*\qquad \mbox{as } p \to \infty.
\]
\end{theorem}
\begin{remark}
The least energy solution is not necessarily the unique positive solution for sufficiently large $p$. Indeed, consider an equilateral Y-shaped metric graph whose three edges have the same length $L$. Let $w_p$ be the positive solution of
\[
-w'' = w^p \quad \mbox{in } (-L,L), \qquad w(-L) = w(L) = 0.
\]
By symmetry, $w_p'(0)=0$. On each edge of the Y-shaped graph, identified with $[0,L]$ from the interior vertex $\bo{v}_0$ to an end vertex, define
\[
\widehat{u}_p(s)=w_p(s).
\]
Then $\widehat{u}_p$ is continuous at $\bo{v}_0$ and satisfies
\[
\sum_{\bo{e} \succ \bo{v}_0} \partial_{\bo{e}}\widehat{u}_p(\bo{v}_0) = 3w_p'(0)=0.
\]
Thus $\widehat{u}_p$ is a positive solution on the Y-shaped graph, and its maximum is attained at the interior vertex $\bo{v}_0$.

On the other hand, let $x_p$ be a maximum point of the least energy solution $u_p$. By Theorem~\ref{main2},
\[
dist(x_p,M) \longrightarrow 0.
\]
Since $\bo{v}_0 \notin M$ by Theorem~\ref{main}, we have
\[
x_p \neq \bo{v}_0
\]
for all sufficiently large $p$. Hence, $\widehat{u}_p$ and $u_p$ are distinct for all sufficiently large $p$. In particular, the problem admits at least two distinct positive solutions for all sufficiently large $p$.
\end{remark}
We next consider a simple example of a $Y$-shaped metric graph. In this case, the Dirichlet--Kirchhoff Green function can be computed explicitly, which allows us to determine the set $M$ and to describe explicitly the asymptotic behavior of the least energy solutions and their maximum points.
\begin{example}
Let $\bo{G}$ be a $Y$-shaped metric graph consisting of three edges $\bo{e}_1$, $\bo{e}_2$, and $\bo{e}_3$ joined at a common interior vertex $\bo{v}_0$. We identify each edge $\bo{e}_i$ with the interval $[0,l_i]$, where $0$ corresponds to $\bo{v}_0$ and $l_i$ corresponds to the end vertex $\bo{v}_i$, for $i=1,2,3$. Without loss of generality, we assume that
\[
l_1 \geq l_2 \geq l_3 > 0.
\]
Let $a\in\bo{e}_1$ and denote by $s\in[0,l_1]$ the distance from $\bo{v}_0$ to $a$. A direct computation of the Dirichlet--Kirchhoff Green function gives
\[
G(a,a) = \frac{(l_1-s)\left( s+\dfrac{l_2l_3}{l_2+l_3} \right)} { l_1+\dfrac{l_2l_3}{l_2+l_3} }.
\]
Since
\[
\frac{l_2l_3}{l_2+l_3}<l_3\leq l_1,
\]
the maximum of $G(a,a)$ on $\bo{e}_1$ is attained at the unique interior point $a_1^*\in\bo{e}_1$ whose distance from $\bo{v}_0$ is
\[
dist(a_1^*,\bo{v}_0) = \frac{1}{2} \left(l_1-\frac{l_2l_3}{l_2+l_3} \right).
\]
Hence,
\[
\max_{a\in\bo{e}_1}G(a,a) = G(a_1^*,a_1^*) = \frac{1}{4} \left( l_1+\frac{l_2l_3}{l_2+l_3}\right).
\]
The same computation applies to the other edges by permuting the indices. Comparing the maximum values on the three edges, we find that the largest value is attained on the longest edges. Therefore,
\[
\max_{x\in\bo{G}}G(x,x) = \frac{1}{4}\left(l_1+\frac{l_2l_3}{l_2+l_3}\right).
\]
In particular, $\bo{v}_0\notin M$. If $l_1>l_2\geq l_3$, then
\[
M=\{a_1^*\}.
\]
If $l_1=l_2>l_3$, then
\[
M=\{a_1^*,a_2^*\}.
\]
Finally, if $l_1=l_2=l_3=L$, then
\[
M=\{a_1^*,a_2^*,a_3^*\},
\]
where, for each $i=1,2,3$,
\[
dist(a_i^*,\bo{v}_0)=\frac{L}{4}.
\]
Therefore, Theorem~\ref{main2} yields
\[
dist\left( x_p,\{a_1^*,a_2^*,a_3^*\} \right) \longrightarrow 0 \qquad \mbox{as }p\to\infty.
\]
\end{example}

We next consider the Lane--Emden problem with Kirchhoff--Neumann boundary conditions:
\begin{equation}
\label{E-L Neumann}
\begin{cases}
-\Delta u_p^{(\bo{e})} = |u_p^{(\bo{e})}|^{p-1}u_p^{(\bo{e})} & \text{for each edge } \bo{e}\in \bo{E},\\
\displaystyle\sum_{\bo{e} \succ \bo{v}} \partial u_p^{(\bo{e})}(\bo{v}) = 0 & \text{for each vertex } \bo{v}\in \bo{V}_{int},\\
\partial u_p^{(\bo{e})}(\bo{v}) = 0& \text{for each vertex } \bo{v}\in \bo{V}_{end},\\
u_p^{(\bo{e})}(\bo{v}) = u_p^{(\bo{e}')}(\bo{v})& \text{if } \bo{e}\succ\bo{v} \text{ and } \bo{e}' \succ \bo{v}.
\end{cases}
\end{equation}
Here, the second line is the Kirchhoff condition, the third line is the Neumann boundary condition, and the last line is the continuity
condition at $\bo{v}$. Integrating the equation in \eqref{E-L Neumann} over $\bo{G}$ and using the Kirchhoff--Neumann boundary conditions, we obtain
\[
\int_{\bo{G}}|u_p|^{p-1}u_p\,dx = 0.
\]
In particular, every nontrivial solution of \eqref{E-L Neumann} changes sign. 

Least energy nodal solutions for pure Neumann problems have been studied in several settings. Parini and Weth \cite{P-W} considered the sublinear Neumann problem and established the existence and qualitative properties of least energy nodal solutions. Salda\~{n}a and Tavares \cite{S-T2018} studied least energy nodal solutions for Hamiltonian elliptic systems with pure Neumann boundary conditions, including the associated scalar problem. More recently, in \cite{S-T2022}, they investigated the scalar pure Neumann Lane--Emden equation, which is the corresponding problem on bounded domains in Euclidean spaces, and studied least energy solutions and their asymptotic behavior with respect to the exponent.

We therefore set
\[
\mathcal{A}_p := \left\{u \in H^1(\bo{G}) \setminus \{0\}\ \middle| \int_{\bo{G}}|u|^{p-1}u\,dx = 0 \right\},
\]
and define
\[
\sigma_p^N := \inf_{u \in \mathcal{A}_p} \sup_{t>0}J_p(tu).
\]
Then, for each $p>1$, there exists a solution $u_p$ such that $J_p(u_p)=\sigma_p^N$. We call $u_p$ a least energy solution. For the Kirchhoff--Neumann problem, we define
\begin{equation}
\label{S_infty^N}
S_\infty^N(\bo{G}) := \inf_{\substack{u \in H^1(\bo{G})\\ \max_{\bo{G}}u=1,\ \min_{\bo{G}}u=-1}} \int_{\bo{G}}|\nabla u|^2 dx.
\end{equation}
We also define the set
\[
\mathcal{D} := \left\{ (a, b) \in \bo{G} \times \bo{G}\ \middle|\ dist(a, b) = \max_{x, y \in \bo{G}}dist(x, y) \right\}.
\]
Large exponent asymptotics for Neumann-type problems have also been studied in the literature. Takahashi \cite{Takahashi2014} considered a two-dimensional elliptic problem with a nonlinear Neumann boundary condition and investigated the asymptotic behavior of its least energy solutions as $p \to \infty$. In particular, he showed that the solutions develop a single peak on the boundary and that the location of the peak is determined by the Green function associated with the corresponding linear problem. The following theorem describes the asymptotic behavior of the least energy solutions for the Kirchhoff--Neumann problem.
\begin{theorem}
\label{main3}
Assume that $\bo{G}$ is a compact metric graph without cycles. For each $p > 1$, let $u_p$ be a least energy solution of \eqref{E-L Neumann}. Then, for every sequence $\{p_n\}$ with $p_n \to \infty$, there exist a subsequence, still denoted by $\{p_n\}$, and a pair $(a_*, b_*) \in \mathcal{D}$ such that
\[
u_{p_n} \longrightarrow u_\infty
\]
strongly in $H^1(\bo{G})$ and uniformly on $\bo{G}$, where $u_\infty$ is a minimizer of the limiting variational problem \eqref{S_infty^N} and satisfies
\[
u_\infty(a_*) = -1, \qquad u_\infty(b_*) = 1.
\]
Moreover, if $\bo{U}$ denotes the unique path joining $a_*$ and $b_*$, then
\[
u_\infty(x) = -1+\frac{2dist(a_*, x)}{dist(a_*, b_*)},
\qquad x\in\bo{U},
\]
and $u_\infty$ is constant on each connected component of $\bo{G} \setminus \bo{U}$.
\end{theorem}
The following theorem describes the asymptotic behavior of the maximum and minimum points of the least energy solutions.
\begin{theorem}
\label{main4}
Assume that $\bo{G}$ is a compact metric graph without cycles. Let $u_p$ be a least energy solution of \eqref{E-L Neumann}, and let $x_p^+, x_p^- \in \bo{G}$ be a maximum point and a minimum point of $u_p$, respectively. Then
\[
dist(x_p^-, x_p^+) \longrightarrow \max_{x,y \in \bo{G}}dist(x,y) \qquad \mbox{as } p \to \infty.
\]
\end{theorem}

The remainder of this paper is organized as follows. Section \ref{sect;2} introduces the variational framework for problem \eqref{metric graph}, together with the optimal Sobolev constant and its relation to least energy solutions. Section \ref{sect;3} is devoted to the limiting variational problem and the Dirichlet--Kirchhoff Green function, culminating in a characterization of the minimizers. Building on these results, we establish the convergence of the optimal Sobolev constants and their minimizers in Section \ref{sect;4}. The proofs of Theorem~\ref{main} and Theorem~\ref{main2} are completed in Section \ref{sect;5}. In Section \ref{sect;6}, we study the local asymptotic behavior of the least energy solutions around their maximum points. By introducing a suitable rescaling, we derive the limiting one-dimensional Liouville equation and identify the limiting profile explicitly.

In Section \ref{sect;7}, we introduce the variational framework for the Kirchhoff--Neumann problem \eqref{E-L Neumann} and the corresponding optimal Sobolev constant. Section \ref{sect;8} is devoted to the limiting variational problem for the Kirchhoff--Neumann problem and the convergence of the optimal Sobolev constants and their minimizers. Finally, the proofs of Theorems~\ref{main3} and \ref{main4} are completed in Section~\ref{sect;9}.

The detailed computation of the Dirichlet--Kirchhoff Green function on the $Y$-shaped metric graph is given in Appendix~\ref{Appendix A}.

\section{Variational structure and the Sobolev constant}\label{sect;2}
In this section, we present the variational formulation of problem \eqref{metric graph}. For each $p >1$, we define the optimal Sobolev constant by
\[
S_p(\bo{G}) := \inf_{H_0^1(\bo{G}) \setminus \{0\}} \frac{\int_{\bo{G}} |\nabla u|^2 dx}{\(\int_{\bo{G}} |u|^{p+1} dx\)^{\frac{2}{p+1}}}.
\]
The optimal constant $S_p(\bo G)$ is indeed attained by a nonnegative function that satisfies the corresponding Euler--Lagrange equation, leading us to the solution to problem \eqref{metric graph}, which reads as follows.
\begin{proposition}
\label{S_p}
For each $p>1$, $S_p(\bo{G})$ is attained by a nonnegative function $v_p \in H_0^1(\bo{G})$. Moreover, if $\|v_p\|_{L^{p+1}(\bo{G})} = 1$, then $v_p$ satisfies
\[
-\Delta v_p = S_p(\bo{G})v_p^p\quad \mbox{in}\quad \bo{G}
\]
in the weak sense. Consequently,
\begin{equation}
\label{u_p}
u_p := S_p(\bo{G})^{\frac{1}{p-1}}v_p
\end{equation}
is a positive solution of \eqref{metric graph}.
\end{proposition}
\begin{proof}
The existence of a nonnegative minimizer follows from the direct method in calculus of variations. Let $v_{p_n}$ be a minimizing sequence for $S_p(\bo{G})$ satisfying $\|v_{p_n}\|_{L^{p+1}(\bo{G})} = 1$.
Since~$v_{p_n}$ is bounded in $H_0^1(\bo{G})$, the compact embedding $H_0^1(\bo{G}) \hookrightarrow C(\bo{G})$ implies that, up  to a subsequence, 
\begin{align*}
\begin{cases}
v_{p_n} \weakto v_p\ &\mbox{weakly in}\ H_0^1(\bo{G}),\\
v_{p_n} \to v_p\ &\mbox{uniformly on}\ \bo{G}.
\end{cases}
\end{align*}
Therefore, $\|v_p\|_{L^{p+1}(\bo{G})} = 1$. By weak lower semicontinuity,
\[
S_p(\bo{G}) = \int_{\bo{G}} |\nabla v_p |^2 dx.
\]
Hence, $v_p$ is a minimizer. By the method of Lagrange multipliers, $v_p$ satisfies
\[
-\Delta v_p = S_p(\bo{G})v_p^p\quad \mbox{in}\quad \bo{G}
\]
in the weak sense. Define $u_p$ by \eqref{u_p}. Then $u_p$ is a positive solution of \eqref{metric graph}.
\end{proof}
The following proposition shows the explicit relation between least energy and the optimal Sobolev constant.
\begin{proposition}
\label{sigma_p}
For each $p>1$, the least energy satisfies
\begin{equation}
\label{sigma_p form}
\sigma_p = \frac{p-1}{2(p+1)}S_p(\bo{G})^{\frac{p+1}{p-1}}.
\end{equation}
Moreover, if $v_p$ is a nonnegative minimizer of $S_p(\bo G)$ satisfying  $\|v_p\|_{L^{p+1}(\bo{G})} = 1$,
then the function
\[
u_p := S_p(\bo{G})^{\frac{1}{p-1}}v_p
\]
attains $\sigma_p$. In particular, $u_p$ is a least energy solution of \eqref{metric graph}. Conversely, if $u_p$ is a least energy solution of \eqref{metric graph}, then
\[
v_p := \frac{u_p}{\| u_p \|_{L^{p+1}(\bo{G})}}
\]
is a nonnegative minimizer of $S_p(\bo{G})$ satisfying $\|v_p\|_{L^{p+1}(\bo{G})} = 1$. Moreover,
\[
u_p = S_p(\bo{G})^{\frac{1}{p-1}}v_p.
\]
\end{proposition}
\begin{proof}
Let $u \in H_0^1(\bo{G}) \setminus \{0\}$. Then
\[
J_p(tu) = \frac{t^2}{2}\int_{\bo{G}} |\nabla u|^2 dx - \frac{t^{p+1}}{p+1}\int_{\bo{G}}|u|^{p+1} dx,\qquad t > 0.
\]
Differentiating with respect to $t$, we obtain
\[
\frac{d}{dt}J_p(tu) = t\int_{\bo{G}} |\nabla u|^2 dx - t^p\int_{\bo{G}}|u|^{p+1} dx.
\]
Hence,  $J_p(tu)$ attains its unique maximum at
\[
t_m := \(\frac{\int_{\bo{G}}|\nabla u|^2 dx}{\int_{\bo{G}}|u|^{p+1} dx} \)^{\frac{1}{p-1}}.
\]
Since $t_m^2 \int_{\bo{G}}|\nabla u|^2 dx = t_m^{p+1}\int_{\bo{G}}|u|^{p+1} dx$, it follows that
\begin{align*}
\sup_{t>0}J_p(tu) = J_p(t_m u) &= \(\frac{1}{2} - \frac{1}{p+1} \)t_m^2 \int_{\bo{G}}|\nabla u|^2 dx\\
&= \frac{p-1}{2(p+1)}\(\frac{\int_{\bo{G}}|\nabla u|^2 dx}{\int_{\bo{G}} |u|^{p+1} dx} \)^{\frac{2}{p-1}}\int_{\bo{G}}|\nabla u|^2 dx\\
&= \frac{p-1}{2(p+1)}\left\{\frac{\int_{\bo{G}}|\nabla u|^2 dx}{\(\int_{\bo{G}}|u|^{p+1} dx \)^{\frac{2}{p+1}}} \right\}^{\frac{p+1}{p-1}}.
\end{align*}
Taking the infimum over $u\in H_0^1(\bo{G})\setminus\{0\}$ yields \eqref{sigma_p form}. Finally, Proposition~\ref{S_p} implies that
\[
J_p(u_p) = \sigma_p.
\]

Conversely, let $u_p$ be a least energy solution of \eqref{metric graph}. Since $u_p$ is a critical point of the restriction of $J_p$ to $H_0^1(\bo{G})$, it satisfies
\[
\int_{\bo{G}}|\nabla u_p|^2 dx = \int_{\bo{G}}|u_p|^{p+1} dx.
\]
Hence,
\[
J_p(u_p) = \frac{p-1}{2(p+1)} \int_{\bo{G}}|\nabla u_p|^2 dx.
\]
Since $J_p(u_p) = \sigma_p$, using \eqref{sigma_p form}, we obtain
\[
\int_{\bo{G}}|\nabla u_p|^2 dx = S_p(\bo{G})^{\frac{p+1}{p-1}}.
\]
Moreover,
\[
\int_{\bo{G}}|u_p|^{p+1} dx = S_p(\bo{G})^{\frac{p+1}{p-1}}.
\]
Therefore, setting
\[
v_p:=\frac{u_p}{\|u_p\|_{L^{p+1}(\bo{G})}},
\]
we obtain $\|v_p\|_{L^{p+1}(\bo{G})} = 1$ and
\[
\int_{\bo{G}}|\nabla v_p|^2 dx = S_p(\bo{G}).
\]
Thus, $v_p$ is a nonnegative minimizer of $S_p(\bo{G})$, which completes the proof.
\end{proof}
The explicit formula established in Proposition~\ref{sigma_p} reduces the asymptotic analysis of least energy solutions to the study of the optimal Sobolev constant. We therefore turn to the corresponding limiting variational problem in the next section.

\section{The limiting variational problem}\label{sect;3}
In this section, we study the limiting variational problem related to the optimal Sobolev constant, discussed in Section \ref{sect;2}, as $p \to \infty$. We first introduce the limiting Sobolev constant:
\[
S_\infty(\bo{G}) := \inf_{\substack{u \in H_0^1(\bo{G}),\\ \|u\|_{L^\infty(\bo{G})} = 1}} \int_{\bo{G}}|\nabla u|^2 dx.
\]
\begin{proposition}
\label{S_infty}
The infimum $S_\infty(\bo{G})$ is attained by a nonnegative function $v_\infty \in H_0^1(\bo{G})$.
\end{proposition}
The proof is similar to that of Proposition \ref{S_p}. 
\begin{proof}
Let $\{v_n\}$ be a minimizing sequence satisfying $\|v_n\|_{L^\infty(\bo{G})} = 1$.
Since the sequence $\{v_n\}$ is bounded in $H_0^1(\bo{G})$, the compact embedding $H_0^1(\bo{G}) \hookrightarrow C(\bo{G})$ preserves the constraint in the limit. The conclusion follows from the weak lower semicontinuity.
\end{proof}
We next establish a basic identity for the Dirichlet--Kirchhoff Green function, which plays a key role in the sequel.
\begin{proposition}
\label{G(a,a)}
For each $a \in \bo{G} \setminus \bo{V}_{end}$, the Dirichlet--Kirchhoff Green function satisfies
\[
\int_{\bo{G}}|\nabla_x G(x, a)|^2 dx = G(a, a) .
\]
\end{proposition}
\begin{proof}
Choosing $\varphi = G(\cdot, a)$ in the weak formulation \eqref{eqn;GF} of the Green function, we obtain
\[
\int_{\bo{G}} |\nabla_x G(x, a)|^2 dx = G(a, a),
\]
which ends the proof.
\end{proof}
We next study the behavior of the diagonal Green function 
\[
R(a):= G(a, a)
\]
near an interior vertex. Using the energy identity in Proposition~\ref{G(a,a)}, we compute the directional derivative of $R$ along each incident edge. This will be used to exclude interior vertices from the set $M$.
\begin{lemma}
\label{directional derivative}
Let $\bo{v} \in \bo{V}_{int}$, and let $\bo{e}$ be an edge incident to $\bo{v}$. For $s > 0$, let $a_s \in \bo{e}$ be the point satisfying
\[
dist(a_s, \bo{v}) = s.
\]
Then 
\[
\lim_{s \downarrow 0}\frac{G(a_s,a_s) - G(\bo{v},\bo{v})}{s} = 1 + 2\partial_{\bo{e}}G(\bo{v},\bo{v}),
\]
where $\partial_{\bo{e}}G(\bo{v},\bo{v})$ denotes the outward derivative of $G(\cdot,\bo{v})$ at $\bo{v}$ along $\bo{e}$.
\end{lemma}
\begin{proof}
For simplicity, we introduce the norm on $H_0^1(\bo{G})$ by
\[
\|u\|_{H_0^1(\bo{G})} := \( \int_{\bo{G}}|\nabla u|^2 dx \)^{1/2}.
\]
By the Poincar\'e inequality, this norm is equivalent to the usual $H^1$-norm on $H_0^1(\bo{G})$. Let
\[
g_s := G(\cdot,a_s),\qquad g_0 := G(\cdot,\bo{v}),\qquad h_s := g_s - g_0.
\]
By Proposition~\ref{G(a,a)},
\[
G(a,a) = \|G(\cdot,a)\|_{H_0^1(\bo{G})}^2.
\]
Hence
\begin{equation}
\label{eqn;G.F. and h_s}
\begin{aligned}
G(a_s,a_s) - G(\bo{v},\bo{v}) &= \|g_s\|_{H_0^1(\bo{G})}^2 - \|g_0\|_{H_0^1(\bo{G})}^2\\
&= 2\langle g_0, h_s\rangle_{H_0^1(\bo{G})} + \|h_s\|_{H_0^1(\bo{G})}^2.
\end{aligned}
\end{equation}
Using the weak formulation of the Green function \eqref{eqn;GF} with $\varphi = h_s$, we obtain
\[
\langle g_0, h_s\rangle_{H_0^1(\bo{G})} = h_s(\bo{v}) = G(a_s,\bo{v}) - G(\bo{v},\bo{v}),
\]
where we used the symmetry of the Green function. Therefore, \eqref{eqn;G.F. and h_s} gives
\begin{equation}
\label{eqn;G.F. expansion}
G(a_s,a_s)-G(\bo{v},\bo{v}) = 2\(G(a_s,\bo{v})-G(\bo{v},\bo{v})\) +\|h_s\|_{H_0^1(\bo{G})}^2.
\end{equation}

Since the pole of $G(\cdot,\bo{v})$ is located in $\bo{v}$, $G(\cdot,\bo{v})$ satisfies
\[
-\frac{d^2}{dt^2}G(t,\bo{v}) = 0
\]
on the interior of the edge $\bo{e}$. Hence $G(\cdot,\bo{v})$ has constant derivative along $\bo{e}$, and
\begin{equation}
\label{eqn;G.F. derivative}
G(a_s,\bo{v})-G(\bo{v},\bo{v}) = s\,\partial_{\bo{e}}G(\bo{v},\bo{v}).
\end{equation}

It remains to estimate $\|h_s\|^2_{H_0^1(\bo{G})}$. By the weak formulation of the Green function \eqref{eqn;GF}, for each $\varphi \in H_0^1(\bo{G})$,
\[
\langle g_s,\varphi\rangle_{H_0^1(\bo{G})} = \varphi(a_s)
\]
and
\[
\langle g_0,\varphi\rangle_{H_0^1(\bo{G})} = \varphi(\bo{v}).
\]
Since $h_s = g_s - g_0$, it follows that
\begin{equation}
\label{eqn;inner product}
\langle h_s,\varphi\rangle_{H_0^1(\bo{G})} = \varphi(a_s) - \varphi(\bo{v}).
\end{equation}
Taking $\varphi=h_s$ in \eqref{eqn;inner product}, we have
\[
\|h_s\|_{H_0^1(\bo{G})}^2 = h_s(a_s) - h_s(\bo{v}).
\]
Since $a_s$ and $\bo{v}$ are connected by an edge segment of length $s$, the fundamental theorem of calculus and the Cauchy--Schwarz inequality give
\[
\|h_s\|_{H_0^1(\bo{G})}^2 = \left| \int_0^s h_s'(t)\,dt \right| \leq s^{1/2} \left(\int_0^s|h_s'(t)|^2\,dt\right)^{1/2} \leq s^{1/2}\|h_s\|_{H_0^1(\bo{G})}.
\]
Thus,
\begin{equation}
\label{eqn;h_s U.B.}
\|h_s\|_{H_0^1(\bo{G})}^2 \leq s.
\end{equation}

For the reverse inequality, fix $0<\delta<|\bo{e}|$ and define $\varphi_s\in H_0^1(\bo{G})$ by
\[
\varphi_s(t)=
\begin{cases}
t, & 0\le t \leq s,\\
\displaystyle s\frac{\delta-t}{\delta-s}, & s \leq t \leq \delta,\\
0, & \delta \geq t,
\end{cases}
\]
on $\bo{e}$, and $\varphi_s=0$ on all other edges. Then
\[
\varphi_s(a_s) - \varphi_s(\bo{v}) = s
\]
and
\[
\|\varphi_s\|_{H_0^1(\bo{G})}^2 = s + \frac{s^2}{\delta-s} = \frac{s\delta}{\delta-s}.
\]
Taking $\varphi=\varphi_s$ in \eqref{eqn;inner product} and using the Cauchy--Schwarz inequality, we obtain
\[
s = \langle h_s,\varphi_s\rangle_{H_0^1(\bo{G})} \leq \|h_s\|_{H_0^1(\bo{G})} \|\varphi_s\|_{H_0^1(\bo{G})}.
\]
Therefore,
\begin{equation}
\label{eqn;h_s L.B.}
\|h_s\|_{H_0^1(\bo{G})}^2 \geq \frac{s^2}{\|\varphi_s\|_{H_0^1(\bo{G})}^2} = s-\frac{s^2}{\delta}.
\end{equation}

Combining \eqref{eqn;h_s U.B.} and \eqref{eqn;h_s L.B.}, we obtain
\[
s-\frac{s^2}{\delta} \leq \|h_s\|_{H_0^1(\bo{G})}^2 \leq s,
\]
and hence
\begin{equation}
\label{eqn;h_s asymp}
\|h_s\|_{H_0^1(\bo{G})}^2 = s+O(s^2) \qquad \mbox{as } s \downarrow 0.
\end{equation}

Finally, substituting \eqref{eqn;G.F. derivative} and \eqref{eqn;h_s asymp} into \eqref{eqn;G.F. expansion}, we obtain
\[
G(a_s,a_s) - G(\bo{v},\bo{v}) = s\(1 + 2\partial_{\bo{e}}G(\bo{v},\bo{v})\) + O(s^2).
\]
Dividing by $s$ and letting $s \downarrow 0$ yields
\[
\lim_{s \downarrow 0} \frac{G(a_s,a_s)-G(\bo{v},\bo{v})}{s} = 1+2\partial_{\bo{e}}G(\bo{v},\bo{v}).
\]
This completes the proof.
\end{proof}
Using the directional derivative formula established in Lemma~\ref{directional derivative}, we show that the diagonal Green function cannot attain its maximum at a vertex.
\begin{proposition}
\label{M no vertex}
The set $M$ satisfies
\[
M\subset \bo{G}\setminus\bo{V}.
\]
\end{proposition}
\begin{proof}
First, let $\bo{v} \in \bo{V}_{end}$. By the Dirichlet boundary condition,
\[
G(\bo{v},\bo{v})=0.
\]
On the other hand, $G(a,a)>0$ for every $a \in \bo{G} \setminus \bo{V}_{end}$ by Proposition~\ref{G(a,a)}. Hence,
\[
\bo{v}\notin M.
\]

We next consider $\bo{v} \in \bo{V}_{int}$. Let
\[
\bo{e}_1,\ldots,\bo{e}_d
\]
be the edges incident to $\bo{v}$, where $d=\deg \bo{v} \geq 3$. We first claim that
\begin{equation}
\label{eqn;Green pole}
\sum_{j = 1}^d \partial_{\bo{e}_j}G(\bo{v},\bo{v}) =-1.
\end{equation}
Indeed, since $G(\cdot,\bo{v})$ satisfies
\[
-\Delta G(\cdot,\bo{v}) = \delta_{\bo{v}},
\]
the weak formulation \eqref{eqn;GF} gives
\[
\int_{\bo{G}}\nabla_xG(x,\bo{v})\nabla\varphi(x) dx = \varphi(\bo{v}) \qquad \mbox{for all } \varphi\in H_0^1(\bo{G}).
\]
Integrating by parts on each edge incident to $\bo{v}$ and using the Kirchhoff conditions at all other interior vertices, we obtain
\[
-\( \sum_{j = 1}^d \partial_{\bo{e}_j}G(\bo{v},\bo{v}) \)\varphi(\bo{v}) = \varphi(\bo{v}),
\]
which yields \eqref{eqn;Green pole}.

For each $j=1,\ldots,d$, let $a_{j, s}\in\bo{e}_j$ satisfy
\[
dist(a_{j,s},\bo{v}) = s.
\]
By Lemma~\ref{directional derivative},
\[
\lim_{s \downarrow 0}\frac{R(a_{j,s}) - R(\bo{v})}{s} = 1 + 2\partial_{\bo{e}_j}G(\bo{v},\bo{v}).
\]
Therefore, summing over $j=1,\ldots,d$ and using \eqref{eqn;Green pole}, we obtain
\begin{align*}
\sum_{j=1}^d \lim_{s \downarrow 0}\frac{R(a_{j,s})-R(\bo{v})}{s}
&= d + 2\sum_{j=1}^d \partial_{\bo{e}_j}G(\bo{v},\bo{v})\\
&= d-2.
\end{align*}
Since $d\ge3$, we have $d - 2 > 0$. Hence, there exists an incident edge $\bo{e}_{j_0}$ such that
\[
\lim_{s \downarrow 0} \frac{R(a_{j_0,s})-R(\bo{v})}{s} >0.
\]
Thus, for sufficiently small $s>0$, $R(a_{j_0,s})>R(\bo{v})$, and consequently
\[
\bo{v}\notin M.
\]

Since every vertex belongs to either $\bo{V}_{end}$ or $\bo{V}_{int}$, we conclude that
\[
M \subset \bo{G} \setminus \bo{V}.
\]
This completes the proof.
\end{proof}
We next establish a maximum property of the Dirichlet--Kirchhoff Green function, which will also be used to study the location of the maximum points of the least energy solutions.
\begin{proposition}\label{maximum; GF}
Let $a \in M$. Then,
\[
G(x, a) \leq G(a, a) \qquad \mbox{for each}\ x \in \bo{G}.
\]
Moreover, equality holds if and only if $x = a$. In particular, $G(\cdot, a)$ attains its maximum uniquely at $a$.
\end{proposition}
\begin{proof}
Let $a \in M$ and $x \in \bo{G}\setminus \bo{V}_{end}$. By the weak formulation \eqref{eqn;GF} and the symmetry of the Green function, we have
\[
G(x,a) = \int_{\bo{G}}\nabla_yG(y,a)\nabla_yG(y,x)dy.
\]
Thus, by the Cauchy--Schwarz inequality and Proposition~\ref{G(a,a)}, we obtain
\begin{align*}
|G(x,a)|^2
&\leq
\left(\int_{\bo{G}}|\nabla_yG(y,a)|^2dy\right)
\left(\int_{\bo{G}}|\nabla_yG(y,x)|^2dy\right)\\
&=G(a,a)G(x,x).
\end{align*}
Since $a \in M$, we have
\[
G(x,x) \leq G(a,a).
\]
Therefore,
\[
|G(x,a)| \leq G(a,a),
\]
and hence
\[
G(x,a) \leq G(a,a).
\]
If $x \in \bo{V}_{end}$, then $G(x,a)=0$ by the Dirichlet boundary condition, and hence the same inequality holds.

We next prove the uniqueness of the maximum point. Suppose that
\[
G(x,a)=G(a,a)
\]
for some $x \in \bo{G}$. Then $x \notin \bo{V}_{end}$, and all the above inequalities must be equalities. In particular,
\[
G(x,x)=G(a,a),
\]
and equality holds in the Cauchy--Schwarz inequality. Hence, $\nabla G(\cdot,x)$ and $\nabla G(\cdot,a)$ are linearly dependent. Moreover,
\[
\|\nabla G(\cdot,x)\|_{L^2(\bo{G})} = \|\nabla G(\cdot,a)\|_{L^2(\bo{G})},
\]
and
\[
\int_{\bo{G}}\nabla_yG(y,a)\nabla_yG(y,x)dy = G(a,a).
\]
Therefore,
\[
\nabla G(\cdot,x)=\nabla G(\cdot,a)
\]
almost everywhere on $\bo{G}$. Since $G(\cdot,x), G(\cdot,a) \in H_0^1(\bo{G})$, it follows that
\[
G(\cdot,x)=G(\cdot,a).
\]
By the weak formulation \eqref{eqn;GF}, for every $\varphi \in H_0^1(\bo{G})$,
\begin{align*}
\varphi(x) &=\int_{\bo{G}}\nabla_yG(y,x)\nabla\varphi(y)dy\\
&=\int_{\bo{G}}\nabla_yG(y,a)\nabla\varphi(y)dy\\
&=\varphi(a).
\end{align*}
Since $H_0^1(\bo{G})$ separates distinct points of $\bo{G}\setminus\bo{V}_{end}$, we conclude that $x=a$. Therefore,
\[
G(x,a)<G(a,a)\qquad \mbox{for every}\ x \in \bo{G}\setminus\{a\}.
\]
This completes the proof.
\end{proof}
We now use the Dirichlet--Kirchhoff Green function to construct a candidate minimizer for the limiting variational problem.
For each $a_*\in M$, define
\begin{equation}
\label{w_a}
w_{a_*}(x):=
\frac{G(x,a_*)}{G(a_*,a_*)}.
\end{equation}
By Proposition~\ref{maximum; GF},
\[
\|w_{a_*}\|_{L^\infty(\bo G)}=1,
\]
and $w_{a_*}$ attains its maximum uniquely at $a_*$.
Moreover, Proposition~\ref{G(a,a)} gives
\[
\int_{\bo G}|\nabla w_{a_*}(x)|^2\,dx = \frac1{G(a_*,a_*)}.
\]
The following proposition shows that this candidate is indeed optimal.
\begin{proposition}
\label{limiting problem}
Let $a_* \in M$. Then
\[
S_\infty(\bo{G}) = \frac{1}{G(a_*,a_*)}.
\]
Moreover $w_{a_*}(x)$ defined in \eqref{w_a} is a minimizer of $S_\infty(\bo{G})$. Conversely, every nonnegative minimizer of $S_\infty(\bo{G})$ is of this form for some $a^*\in\mathcal{M}$.
\end{proposition}
\begin{proof}
Let $a_* \in M$. Recall that the normalized Green function
\[
w_{a_*}(x) := \frac{G(x,a_*)}{G(a_*,a_*)}
\]
satisfies  $\|w_{a_*}\|_{L^\infty(\bo{G})} = 1$ and 
\[
\int_{\bo{G}} |\nabla w_{a_*}(x)|^2 dx = \frac{1}{G(a_*,a_*)}.
\]
Hence, $w_{a_*}$ is admissible in the minimizing problem defining $S_\infty(\bo{G})$, and therefore
\begin{equation}
S_\infty(\bo{G}) \leq \frac{1}{G(a_*, a_*)}.
\label{eqn;upper}
\end{equation}
We next prove the reverse inequality. Let $u \in H_0^1(\bo{G})$ satisfy $\|u\|_{L^\infty(\bo{G})} = 1$. Since $\bo{G}$ is compact and $u$ is continuous on $\bo{G}$, there exists $b \in \bo{G}$ such that
\[
|u(b)| = 1.
\]
By the weak formulation \eqref{eqn;GF} of the Green function,
\[
u(b) = \int_{\bo{G}}\nabla_x G(x,b) \nabla u(x) dx.
\]
Thus, by the Cauchy--Schwarz inequality,
\[
1 = |u(b)|^2 \leq \(\int_{\bo{G}}|\nabla_x G(x,b)|^2 dx \)\(\int_{\bo{G}}|\nabla u|^2 dx \)
\]
Using Proposition~\ref{G(a,a)}, we obtain
\[
1 \leq G(b,b)\int_{\bo{G}}|\nabla u|^2 dx.
\]
Since $a^*\in\mathcal{M}$, we have $G(b, b) \leq G(a^*, a^*)$.
Therefore,
\[
\frac{1}{G(a^*,a^*)} \leq \frac{1}{G(b, b)} \leq \int_{\bo{G}}|\nabla u|^2 dx.
\]
Taking the infimum over all admissible functions $u$, we obtain
\begin{equation}
\label{eqn;lower}
\frac{1}{G(a^*,a^*)} \leq S_\infty(\bo{G}).
\end{equation}
Combining the upper \eqref{eqn;upper} and lower \eqref{eqn;lower} estimates gives
\[
S_\infty(\bo{G}) = \frac{1}{G(a_*,a_*)}.
\]
Since equality is attained by $w_{a_*}$, it follows that $w_{a_*}$ is a minimizer of $S_\infty(\bo G)$.

Conversely, let $v_\infty$ be a nonnegative minimizer of $S_\infty(\bo{G})$. Since $\|v_\infty\|_{L^\infty(\bo{G})} = 1$, there exists $b \in \bo{G}$ such that
\[
v_\infty(b) = 1.
\]
By the weak formulation of the Green function and the Cauchy--Schwarz inequality,
\[
1 = v_\infty(b)^2 \leq G(b,b)\int_{\bo{G}}|\nabla v_\infty|^2 dx.
\]
Since $v_\infty$ is a minimizer and $a^* \in \mathcal{M}$, we have
\[
\int_{\bo{G}}|\nabla v_\infty|^2 dx = S_\infty(\bo{G}) = \frac{1}{G(a^*,a^*)}.
\]
Hence,
\[
1 \leq \frac{G(b,b)}{G(a^*,a^*)} \leq 1.
\]
Therefore,
\[
G(b,b) = G(a^*,a^*),
\]
and hence $b\in\mathcal{M}$. Moreover, equality holds in the Cauchy--Schwarz inequality. Thus,
\[
v_\infty = \frac{G(\cdot,b)}{G(b,b)}.
\]
\end{proof}
We point out that $S_\infty(\bo{G})$ is also the best constant in the Sobolev inequality
\[
S_\infty(\bo{G}) \|u\|_{L^\infty(\bo{G})}^2 \leq \int_{\bo{G}} |\nabla u|^2 dx, \qquad u \in H_0^1(\bo{G}).
\]

\section{Convergence of the optimal Sobolev constants}\label{sect;4}
In this section, we establish the convergence of the optimal Sobolev constants and the asymptotic behavior of their normalized minimizers. These results provide the key ingredients for the proof of Theorem~\ref{main}.
\begin{proposition}
\label{S_p to S_infty}
As $p \to \infty$, one has
\begin{equation}
\label{limit}
S_p(\bo{G}) \to S_\infty(\bo{G}).
\end{equation}
\end{proposition}
\begin{proof}
We first prove the upper bound. Let $a_* \in M$, and let $w_{a_*}$ be defined by \eqref{w_a}. By Proposition~\ref{limiting problem}, $w_{a_*}$ is a minimizer of $S_\infty(\bo{G})$, and hence
\[
\|w_{a_*}\|_{L^\infty(\bo{G})} = 1,\qquad \int_{\bo{G}}|\nabla w_{a_*}|^2 dx = S_\infty(\bo{G}).
\]
Using $w_{a_*}$ as a test function in the definition of $S_p(\bo{G})$, we obtain
\[
S_p(\bo{G}) \leq \frac{\int_{\bo{G}}|\nabla w_{a_*}|^2 dx}{\(\int_{\bo{G}}|w_{a_*}|^{p+1} dx\)^{\frac{2}{p+1}}}.
\]
Since $\|w_{a_*}\|_{L^{p+1}(\bo{G})} \to \|w_{a_*}\|_{L^\infty(\bo{G})} = 1$ as $p \to \infty$, it follows that
\begin{equation}
\label{U. E.}
\limsup_{p \to \infty}S_p(\bo{G}) \leq S_\infty(\bo{G}).
\end{equation}

We next prove the lower bound. For each $p > 1$, let $v_p$ be a normalized nonnegative minimizer of $S_p(\bo{G})$. Then,
\[
\int_{\bo{G}}|\nabla v_p|^2 dx = S_p(\bo{G}).
\]
The above estimate implies that $v_p$ is bounded in $H_0^1(\bo{G})$. Hence, up to a subsequence, there exists $v_\infty \in H_0^1(\bo{G})$ such that
\begin{align*}
\begin{cases}
v_p \weakto v_\infty\ &\mbox{weakly in}\ H_0^1(\bo{G}),\\
v_p \to v_\infty\ &\mbox{uniformly on}\ \bo{G}.
\end{cases}
\end{align*}

Now we claim that $\|v_\infty\|_{L^\infty(\bo{G})} = 1$. Let $|\bo{G}| := \sum\limits_{\bo{e} \in \bo{E}}l(\bo{e})$. Then it follows that
\[
1 = \|v_p\|_{L^{p+1}(\bo{G})} \leq |G|^{\frac{1}{p+1}} \|v_p\|_{L^\infty(\bo{G})}.
\]
Therefore, the uniform convergence yields that
\[
1 \leq \|v_\infty\|_{L^\infty(\bo{G})}.
\]

On the other hand, suppose that $\|v_\infty\|_{L^\infty(\bo{G})} > 1$. By continuity, there exists a subset $\bo{U} \subset \bo{G}$ of positive measure and a constant $\delta > 0$ such that
\[
v_\infty(x) \geq 1 + \delta \qquad \mbox{for each}\ x \in \bo{U}.
\]
By uniform convergence, for sufficiently large $p$,
\[
v_p(x) \geq 1 + \frac{\delta}{2}\qquad \mbox{for each}\ x \in \bo{U}.
\]
Consequently,
\[
\int_{\bo{G}} |v_p|^{p+1} dx \geq \int_{\bo{U}} |v_p|^{p+1} dx \geq |\bo{U}|\(1 + \frac{\delta}{2} \)^{p+1} \to +\infty
\]
as $p \to \infty$. This contradicts $\|v_p\|_{L^{p+1}(\bo{G})} = 1$. Thus, $\|v_\infty\|_{L^\infty(\bo{G})} = 1$.

Since $v_\infty$ is admissible for $S_\infty(\bo{G})$, by the weak lower semicontinuity, we have
\begin{equation}
\label{L. E.}
S_\infty(\bo{G}) \leq \int_{\bo{G}} |\nabla v_\infty|^2 dx \leq \liminf_{p \to \infty}\int_{\bo{G}} |\nabla v_p|^2 dx = \liminf_{p \to \infty}S_p(\bo{G}).
\end{equation}
Combining the upper \eqref{U. E.} and lower \eqref{L. E.} estimates, we conclude that \eqref{limit}.
\end{proof}
Having established the convergence of the optimal Sobolev constants, we now turn to the asymptotic behavior of the corresponding  normalized minimizers. The next proposition shows that, up to a subsequence, these minimizers converge to a minimizer of the limiting problem, which is characterized by a normalized Dirichlet--Kirchhoff Green function.
\begin{proposition}
\label{normalized v_p}
Let $v_p$ be a normalized nonnegative minimizer of $S_p(\bo{G})$. Then, for each sequence $\{p_n\}$ satisfying $p_n \to \infty$, there exists a subsequence, still denoted by $\{p_n\}$, and a point $a_* \in M$ such that
\[
v_{p_n} \longrightarrow \frac{G(\cdot, a_*)}{G(a_*,a_*)}
\]
strongly in $H_0^1(\bo{G})$ and uniformly on $\bo{G}$.
\end{proposition}
\begin{proof}
By the proof of Proposition~\ref{S_p to S_infty}, after passing to a subsequence, there exists $v_\infty \in H_0^1(\bo{G})$ such that
\begin{align*}
\begin{cases}
v_{p_n} \weakto v_\infty\ &\mbox{weakly in}\ H_0^1(\bo{G}),\\
v_{p_n} \to v_\infty\ &\mbox{uniformly on}\ \bo{G}.
\end{cases}
\end{align*}
Moreover, $\|v_\infty\|_{L^\infty(\bo{G})} = 1$. By the weak lower semicontinuity and the convergence $S_{p_n}(\bo{G}) \to S_\infty(\bo{G})$, we have
\[
S_\infty(\bo{G}) \leq \int_{\bo{G}}|\nabla v_\infty|^2 dx \leq \liminf_{n \to \infty}S_{p_n}(\bo{G}) = S_\infty(\bo{G}).
\]
Hence, $v_\infty$ is a minimizer of $S_\infty(\bo{G})$.

By the characterization of the minimizer of $S_\infty(\bo{G})$, there exists $a_* \in M$ such that
\[
v_{\infty} = \frac{G(\cdot,a_*)}{G(a_*,a_*)}.
\]
Furthermore, 
\[
\int_{\bo{G}}|\nabla v_{p_n}|^2 dx = S_{p_n}(\bo{G}) \to S_\infty(\bo{G}) = \int_{\bo{G}}|\nabla v_\infty|^2 dx.
\]
The weak convergence together with the convergence of the norms implies
\[
v_{p_n} \to v_\infty\ \mbox{strongly in}\ H_0^1(\bo{G}).
\]
The uniform convergence was already obtained above.
\end{proof}

\section{Proofs of Theorem~\ref{main} and Theorem~\ref{main2}}\label{sect;5}
In this section, we complete the proofs of Theorem~\ref{main} and Theorem~\ref{main2} by combining the convergence result obtained in the previous section with the relation between the normalized Sobolev minimizers and the least energy solution.
\begin{proof}[Proof of Theorem~\ref{main}]
By Proposition~\ref{M no vertex}, we have
\[
M\subset\bo{G}\setminus\bo{V}.
\]

Let $\{p_n\}$ be an arbitrary sequence satisfying $p_n \to \infty$. By Proposition~\ref{sigma_p}, the function
\[
v_p := \frac{u_p}{\|u_p\|_{L^{p+1}(\bo{G})}}
\]
is a nonnegative minimizer of $S_p(\bo{G})$ satisfying $\|v_p\|_{L^{p+1}(\bo{G})}=1$. By Proposition~\ref{normalized v_p}, there exists a subsequence, still denoted by $\{p_n\}$, and a point $a_* \in M$ such that
\[
v_{p_n} \longrightarrow \frac{G(\cdot, a_*)}{G(a_*,a_*)}
\]
strongly in $H_0^1(\bo{G})$ and uniformly on $\bo{G}$.

Moreover, by Proposition~\ref{sigma_p}, the least energy solution $u_p$ is given by $u_p = S_p(\bo{G})^{\frac{1}{p-1}}v_p$.\\
Since $S_p(\bo{G}) \to S_\infty(\bo{G})>0$, we have $S_p(\bo{G})^{\frac{1}{p-1}} \to 1$. Therefore
\[
u_{p_n} \longrightarrow \frac{G(\cdot, a_*)}{G(a_*,a_*)}
\]
strongly in  $H_0^1(\bo{G})$ and uniformly on $\bo{G}$. This completes the proof.
\end{proof}
We next prove Theorem~\ref{main2} by combining the uniform convergence in Theorem~\ref{main} with the maximum property of the Green function established in Proposition~\ref{maximum; GF}.
\begin{proof}[Proof of Theorem~\ref{main2}]
Let $\{p_n\}$ be an arbitrary sequence satisfying $p_n\to\infty$. Since $\bo{G}$ is compact, the sequence $\{x_{p_n}\}$ admits a convergent subsequence. Let $x_*$ be an arbitrary accumulation point of $\{x_{p_n}\}$, and, after passing to a subsequence, assume that
\[
x_{p_n}\to x_*.
\]
By Theorem~\ref{main}, after passing to a further subsequence, there exists $a_*\in M$ such that
\[
u_{p_n}\to \frac{G(\cdot,a_*)}{G(a_*,a_*)}
\]
uniformly on $\bo{G}$. Since $x_{p_n}$ is a maximum point of $u_{p_n}$, we have
\[
u_{p_n}(x_{p_n})\geq u_{p_n}(a_*).
\]
Passing to the limit, we obtain
\[
\frac{G(x_*,a_*)}{G(a_*,a_*)} \geq \frac{G(a_*,a_*)}{G(a_*,a_*)} = 1.
\]
On the other hand, Proposition~\ref{maximum; GF} gives
\[
G(x_*,a_*)\leq G(a_*,a_*).
\]
Hence,
\[
G(x_*,a_*)=G(a_*,a_*).
\]
Again by Proposition~\ref{maximum; GF}, the maximum point of
$G(\cdot,a_*)$ is unique, and therefore
\[
x_*=a_*\in M.
\]
Thus, every accumulation point of $\{x_{p_n}\}$ belongs to $M$. Since $\bo{G}$ is compact, it follows that
\[
dist(x_{p_n},M)\to0.
\]
Since $\{p_n\}$ was arbitrary, we conclude that
\[
dist(x_p,M)\to0 \qquad \mbox{as } p\to\infty.
\]
Finally, if $M=\{a_*\}$, then every accumulation point of $\{x_p\}$ is equal to $a_*$. Hence,
\[
x_p\to a_* \qquad \mbox{as } p\to\infty.
\]
\end{proof}

\section{Local asymptotic behavior around maximum points}\label{sect;6}
In this section, we study the local asymptotic behavior of the least energy solution $u_p$ around its maximum points as $p \to \infty$. The rescaling analysis for large-exponent problems goes back to Adimurthi--Grossi \cite{A-G}, and related one-dimensional asymptotic results were obtained by Takahashi \cite{Takahashi}. By Theorem~\ref{main2} and Proposition~\ref{M no vertex}, the maximum points of $u_p$ approach the set $M$, which is contained in the interior of the edges. Therefore, for sufficiently large $p$, the behavior of $u_p$ near its maximum points can be regarded locally as a one-dimensional problem. Motivated by the above results, we introduce a suitable rescaling around a maximum point of $u_p$ and derive the corresponding limiting profile on $\mathbb{R}$.

Let $x_p \in \bo{G}$ be a maximum point of $u_p$. For sufficiently large $p$, let $\bo{e}_p$ be the edge containing $x_p$ in its interior. We identify $\bo{e}_p$ with an interval and denote by $d_p^-$ and $d_p^+$ the distances from $x_p$ to the two endpoints of $\bo{e}_p$, respectively. Define $\varepsilon_p > 0$ and $\widetilde{u}_p$ by
\begin{equation}
\label{eqn;rescaling}
\left\{
\begin{aligned}
&p\varepsilon_p^2 \|u_p\|_{L^\infty(\bo{G})}^{p-1} = 1,\\
&\widetilde{u}_p(t) = \frac{p}{\|u_p\|_{L^\infty(\bo{G})}} \left\{ u_p(x_p+\varepsilon_p t) - \|u_p\|_{L^\infty(\bo{G})} \right\},\\
&t\in I_p := \( -\frac{d_p^-}{\varepsilon_p}, \frac{d_p^+}{\varepsilon_p} \).
\end{aligned}
\right.
\end{equation}
The following theorem describes the limiting profile of the rescaled least energy solutions.
\begin{theorem}
Let $\widetilde{u}_p$ be defined by \eqref{eqn;rescaling}. Then
\[
\varepsilon_p \longrightarrow 0
\]
and
\[
\widetilde{u}_p(t) \longrightarrow U(t) := \log \frac{4e^{\sqrt{2}t}}{\(1+e^{\sqrt{2}t}\)^2} \quad \mbox{in }C_{loc}^1(\mathbb{R})
\]
as $p\to\infty$.
\end{theorem}
\begin{proof}
We first show that $\varepsilon_p \to 0$ as $p \to \infty$. By Proposition~\ref{sigma_p}, we have
\[
u_p = S_p(\bo{G})^{\frac{1}{p-1}}v_p, \qquad \|v_p\|_{L^{p+1}(\bo{G})} = 1,
\]
where $v_p$ is a nonnegative minimizer of $S_p(\bo{G})$. Hence,
\[
1 = \int_{\bo{G}}v_p^{p+1} dx \leq |\bo{G}| \|v_p\|_{L^\infty(\bo{G})}^{p+1},
\]
and therefore
\[
\|v_p\|_{L^\infty(\bo{G})} \geq |\bo{G}|^{-\frac{1}{p+1}}.
\]
Consequently,
\begin{align*}
p\|u_p\|_{L^\infty(\bo{G})}^{p-1} &= p S_p(\bo{G}) \|v_p\|_{L^\infty(\bo{G})}^{p-1}\\
&\geq p S_p(\bo{G}) |\bo{G}|^{-\frac{p-1}{p+1}}.
\end{align*}
By Proposition~\ref{S_p to S_infty},
\[
S_p(\bo{G}) \longrightarrow S_\infty(\bo{G})>0.
\]
Thus,
\[
p\|u_p\|_{L^\infty(\bo{G})}^{p-1} \longrightarrow + \infty.
\]
It follows from \eqref{eqn;rescaling} that
\[
\varepsilon_p \longrightarrow 0.
\]
By Proposition~\ref{M no vertex}, we have
\[
M \subset \bo{G}\setminus\bo{V}.
\]
Since $M$ is compact and $\bo{V}$ is finite, it follows that
\[
dist(M,\bo{V})>0.
\]
Moreover, Theorem~\ref{main2} gives
\[
dist(x_p,M) \longrightarrow 0.
\]
Hence, there exist $\delta>0$ and $p_0>1$ such that
\[
dist(x_p,\bo{V}) \geq \delta \qquad \mbox{for all } p\ge p_0.
\]
In particular,
\[
d_p^- \geq \delta, \qquad d_p^+ \geq \delta
\]
for all sufficiently large $p$. Since $\varepsilon_p \to 0$, we obtain
\[
\frac{d_p^-}{\varepsilon_p}\longrightarrow +\infty, \qquad \frac{d_p^+}{\varepsilon_p} \longrightarrow +\infty.
\]
Therefore, for every compact interval $K \subset \mathbb{R}$,
\[
K\subset I_p
\]
for all sufficiently large $p$. By \eqref{eqn;rescaling}, we have
\[
u_p(x_p+\varepsilon_p t) = \|u_p\|_{L^\infty(\bo{G})} \( 1 + \frac{\widetilde{u}_p(t)}{p} \).
\]
Since $u_p$ satisfies
\[
-u_p''=u_p^p
\]
on the edge containing $x_p$, it follows that
\begin{align*}
-\widetilde{u}_p''(t) &= -\frac{p\varepsilon_p^2}{\|u_p\|_{L^\infty(\bo{G})}} u_p''(x_p + \varepsilon_p t)\\
&= p\varepsilon_p^2 \|u_p\|_{L^\infty(\bo{G})}^{p-1}\( 1 + \frac{\widetilde{u}_p(t)}{p}\)^p.
\end{align*}
Hence, by \eqref{eqn;rescaling},
\[
-\widetilde{u}_p''(t) = \( 1 + \frac{\widetilde{u}_p(t)}{p} \)^p \qquad \mbox{in } I_p.
\]
Moreover, since $x_p$ is a maximum point of $u_p$,
\[
\widetilde{u}_p(0) = 0, \qquad \widetilde{u}_p'(0) = 0,
\]
and
\[
-p<\widetilde{u}_p(t) \leq 0 \qquad\mbox{for } t \in I_p.
\]
Then, we have
\[
0 < \( 1 + \frac{\widetilde{u}_p(t)}{p} \)^p \leq 1.
\]
Therefore,
\[
| \widetilde{u}_p''(t)| \leq 1 \qquad \mbox{for } t \in I_p.
\]
Since $\widetilde{u}_p'(0) = 0$, it follows that
\[
|\widetilde{u}_p'(t)| \leq | t |.
\]
Moreover, using $\widetilde{u}_p(0) = 0$, we obtain
\[
|\widetilde{u}_p(t)| \leq \frac{t^2}{2}.
\]

Let $R>0$. Since every compact interval is contained in $I_p$ for all sufficiently large $p$, the above estimates imply
\[
\|\widetilde{u}_p\|_{L^\infty(-R,R)} \leq \frac{R^2}{2}, \qquad \|\widetilde{u}_p'\|_{L^\infty(-R,R)} \leq R, \qquad \|\widetilde{u}_p''\|_{L^\infty(-R,R)} \leq 1.
\]
Hence, by the Arzel\`a--Ascoli theorem and a diagonal argument, there exist a subsequence, still denoted by $\{\widetilde{u}_p\}$, and a function $U \in C^1(\mathbb{R})$ such that
\[
\widetilde{u}_p \longrightarrow U \quad \mbox{in }C_{loc}^1(\mathbb{R}).
\]
Since $\widetilde{u}_p \to U$ uniformly on $[-R,R]$, the sequence $\{\widetilde{u}_p\}$ is uniformly bounded on $[-R,R]$. Hence,
\[
p\log \( 1 + \frac{\widetilde{u}_p}{p} \) = \widetilde{u}_p + o(1)
\]
uniformly on $[-R,R]$. Therefore,
\[
\( 1 + \frac{\widetilde{u}_p}{p} \)^p \longrightarrow e^U
\]
uniformly on $[-R,R]$. It follows from the equation for $\widetilde{u}_p$ that
\[
-U''=e^U \qquad \mbox{in } \mathbb{R}.
\]
Moreover,
\[
U(0)=0, \qquad U'(0)=0.
\]
Thus $U$ satisfies
\[
\left\{
\begin{aligned}
-U''&=e^U &&\mbox{in } \mathbb{R},\\
U(0)&= 0,\\
U'(0)&= 0.
\end{aligned}
\right.
\]
By the uniqueness of the solution of this initial value problem,
\[
U(t) = \log \frac{4e^{\sqrt{2}t}} {\( 1 + e^{\sqrt{2}t} \)^2}.
\]
Since every convergent subsequence has the same limit $U$, the whole sequence converges to $U$ in $C_{\mathrm{loc}}^1(\mathbb{R})$.
\end{proof}

\section{The Kirchhoff--Neumann problem}\label{sect;7}
In this section, we study the variational structure of the Kirchhoff--Neumann problem \eqref{E-L Neumann}. Integrating the equation in \eqref{E-L Neumann} over $\bo{G}$ and using the Kirchhoff--Neumann boundary conditions, we obtain
\begin{equation}
\label{Neumann constraint}
\int_{\bo{G}} |u|^{p-1}u\, dx=0.
\end{equation}
In particular, every nontrivial solution of \eqref{E-L Neumann} changes sign. For a function $u$, we denote its positive and negative parts by
\[
u^+ := \max\{u, 0\}, \qquad u^- := \max\{-u, 0\},
\]
respectively. For each $p>1$, we also define the optimal Sobolev constant by
\[
S_p^N(\bo{G}) := \inf_{u \in \mathcal{A}_p} \frac{\displaystyle\int_{\bo{G}}|\nabla u|^2 dx}{\displaystyle \(\int_{\bo{G}}|u|^{p+1} dx\)^{\frac{2}{p+1}}}.
\]
The optimal constant $S_p^N(\bo{G})$ is indeed attained by a function that satisfies the corresponding Euler--Lagrange equation, leading us to a solution to the Kirchhoff--Neumann problem, which reads as follows.
\begin{proposition}
\label{S_p Neumann}
For each $p>1$, $S_p^N(\bo{G})$ is attained by a function $v_p\in\mathcal{A}_p$. Moreover, if $\|v_p\|_{L^{p+1}(\bo{G})} = 1$, then $v_p$ satisfies
\[
-\Delta v_p = S_p^N(\bo{G})|v_p|^{p-1}v_p \quad\text{in }\bo{G}
\]
in the weak sense. Consequently,
\begin{equation}
\label{u_p Neumann}
u_p := S_p^N(\bo{G})^{\frac{1}{p-1}}v_p
\end{equation}
is a solution of the Kirchhoff--Neumann problem.
\end{proposition}
\begin{proof}
The existence of a minimizer follows as in the proof of Proposition~\ref{S_p}. We only note that the additional constraint \eqref{Neumann constraint} is preserved under uniform convergence.

It remains to derive the Euler--Lagrange equation. By the method of
Lagrange multipliers, there exist $\lambda_p, \mu_p \in \mathbb{R}$ such
that
\[
\int_{\bo{G}}\nabla v_p \nabla\varphi\,dx = \lambda_p \int_{\bo{G}}|v_p|^{p-1}v_p\varphi\,dx + \mu_p \int_{\bo{G}}|v_p|^{p-1}\varphi\,dx
\]
for every $\varphi\in H^1(\bo{G})$. Taking $\varphi=1$, we obtain $\mu_p=0$, while taking $\varphi=v_p$ yields
\[
\lambda_p = S_p^N(\bo{G}).
\]
Therefore,
\[
-\Delta v_p = S_p^N(\bo{G})|v_p|^{p-1}v_p
\]
in the weak sense. Defining $u_p$ as in \eqref{u_p Neumann}, we obtain a solution of \eqref{E-L Neumann}.
\end{proof}
The following proposition shows the explicit relation between least energy and the optimal Sobolev constant.
\begin{proposition}
\label{sigma_p^N form}
For each $p>1$, the least energy satisfies
\begin{equation}
\label{sigma_p^N relation}
\sigma_p^N = \frac{p-1}{2(p+1)} S_p^N(\bo{G})^{\frac{p+1}{p-1}}.
\end{equation}
Moreover, if $v_p$ is a minimizer of $S_p^N(\bo{G})$ satisfying $\|v_p\|_{L^{p+1}(\bo{G})}=1$, then the function $u_p$ defined by \eqref{u_p Neumann} attains $\sigma_p^N$. In particular, $u_p$ is a least energy solution of \eqref{E-L Neumann}. Conversely, if $u_p$ is a least energy solution of \eqref{E-L Neumann}, then
\[
v_p:=\frac{u_p}{\|u_p\|_{L^{p+1}(\bo{G})}}
\]
is a minimizer of $S_p^N(\bo{G})$ satisfying $\|v_p\|_{L^{p+1}(\bo{G})} = 1$, and
\[
u_p=S_p^N(\bo{G})^{\frac{1}{p-1}}v_p.
\]
\end{proposition}
\begin{proof}
Since $\mathcal{A}_p$ is invariant under multiplication by positive constants, the same argument as in the proof of Proposition~\ref{sigma_p} gives \eqref{sigma_p^N relation}. The last statement follows from Proposition~\ref{S_p Neumann}.

Conversely, let $u_p$ be a least energy solution of \eqref{E-L Neumann}. Testing the equation with $u_p$, we obtain
\[
\int_{\bo{G}}|\nabla u_p|^2 dx = \int_{\bo{G}}|u_p|^{p+1}dx.
\]
Hence
\[
\sigma_p^N = J_p(u_p) = \frac{p-1}{2(p+1)} \int_{\bo{G}}|u_p|^{p+1}dx.
\]
Combining this with \eqref{sigma_p^N relation}, we have
\[
\int_{\bo{G}}|u_p|^{p+1}dx = S_p^N(\bo{G})^{\frac{p+1}{p-1}}.
\]
Therefore,
\[
\|u_p\|_{L^{p+1}(\bo{G})} = S_p^N(\bo{G})^{\frac{1}{p-1}}.
\]
Define
\[
v_p := \frac{u_p}{\|u_p\|_{L^{p+1}(\bo{G})}}.
\]
Then $\|v_p\|_{L^{p+1}(\bo{G})} = 1$. Moreover, since $u_p$ satisfies \eqref{Neumann constraint}, we have $v_p \in \mathcal{A}_p$. Using the identity above, we obtain
\[
\int_{\bo{G}}|\nabla v_p|^2 dx = \frac{\displaystyle\int_{\bo{G}}|\nabla u_p|^2 dx}{\|u_p\|_{L^{p+1}(\bo{G})}^2} = S_p^N(\bo{G}).
\]
Thus $v_p$ is a minimizer of $S_p^N(\bo{G})$. Finally,
\[
u_p = S_p^N(\bo{G})^{\frac{1}{p-1}}v_p.
\]

\end{proof}

\section{Asymptotic behavior of the optimal Sobolev constants for the Kirchhoff--Neumann problem}\label{sect;8}
In this section, we study the limiting variational problem related to the optimal Sobolev constant introduced in the previous section as $p \to \infty$. The explicit formula established in Proposition~\ref{sigma_p^N form} reduces the asymptotic analysis of least energy solutions to the study of the optimal Sobolev constant. We therefore turn to the corresponding limiting variational problem. We define the limiting Sobolev constant by
\begin{equation}
\label{S_infty Neumann}
S_\infty^N(\bo{G}) := \inf_{\substack{u \in H^1(\bo{G})\\ \max_{\bo{G}}u = 1,\ \min_{\bo{G}}u = -1}} \int_{\bo{G}}|\nabla u|^2 dx.
\end{equation}
\begin{proposition}
\label{S_infty Neumann minimizer}
The infimum $S_\infty^N(\bo{G})$ is attained by a function $v_\infty \in H^1(\bo{G})$.
\end{proposition}
\begin{proof}
The proof is similar to that of Proposition~\ref{S_infty}. Indeed, if $\{v_n\}$ is a minimizing sequence for $S_\infty^N(\bo{G})$, then
\[
-1\leq v_n \leq 1 \quad\text{on }\bo{G},
\]
and hence $\{v_n\}$ is bounded in $H^1(\bo{G})$. The compact embedding $H^1(\bo{G})\hookrightarrow C(\bo{G})$ preserves the constraints in \eqref{S_infty Neumann}. The conclusion follows from the weak lower semicontinuity.
\end{proof}
We next characterize the limiting Sobolev constant in terms of the distance on the graph.
\begin{proposition}
\label{S_infty Neumann distance}
Assume that $\bo{G}$ is a compact metric graph without cycles. Then
\begin{equation}
\label{S_infty Neumann formula}
S_\infty^N(\bo{G}) = \frac{4} {\displaystyle \max_{x, y \in\bo{G}} dist(x, y)}.
\end{equation}
\end{proposition}
\begin{proof}
We first prove the lower bound. Let $u \in H^1(\bo{G})$ be admissible for \eqref{S_infty Neumann}. Choose $x_+, x_- \in \bo{G}$ such that
\[
u(x_+) = 1, \qquad u(x_-) = -1.
\]
Since $\bo{G}$ has no cycles, there exists a unique path joining $x_-$ and $x_+$. We denote this path by $\bo{U}$.  Applying the Cauchy--Schwarz inequality along this path, we obtain
\begin{align*}
4 = |u(x_+) - u(x_-)|^2 &\leq dist(x_-, x_+) \int_{\bo{U}} |\nabla u|^2 dx\\
&\leq dist(x_-, x_+) \int_{\bo{G}}|\nabla u|^2 dx.
\end{align*}
Since $dist(x_-, x_+) \leq \max\limits_{x, y \in \bo{G}}dist(x ,y)$, it follows that
\[
\int_{\bo{G}}|\nabla u|^2 dx \geq \frac{4} {\displaystyle\max_{x, y \in \bo{G}}dist(x,y)}.
\]
Taking the infimum over all admissible $u$, we obtain
\[
S_\infty^N(\bo{G}) \geq \frac{4} {\displaystyle\max_{x,y \in \bo{G}}dist(x, y)}.
\]

We next prove the upper bound. Let $a, b \in \bo{G}$ satisfy
\[
dist(a ,b) = \max_{x, y \in \bo{G}}dist(x ,y),
\]
and denote by $\bo{U}_{a,b}$ the unique path joining $a$ and $b$. We define $w \in H^1(\bo{G})$ by
\[
w(x) = -1 + \frac{2dist(a,x)}{dist(a,b)}, \qquad x \in \bo{U}_{a, b},
\]
and extend $w$ constantly to each component of $\bo{G}\setminus\bo{U}_{a, b}$ with the value of $w$ at its point of attachment to $\bo{U}_{a, b}$. Then
\[
\max_{\bo{G}} w = 1, \qquad \min_{\bo{G}}w = -1,
\]
and hence $w$ is admissible for \eqref{S_infty Neumann}. Moreover,
\[
\int_{\bo{G}}|\nabla w|^2 dx = \frac{4}{dist(a, b)} = \frac{4}{\displaystyle\max_{x, y \in \bo{G}}dist(x, y)}.
\]
Therefore,
\[
S_\infty^N(\bo{G}) \leq \frac{4}{\displaystyle\max_{x,y \in \bo{G}}dist(x, y)}.
\]
Combining the upper and lower bounds, we obtain \eqref{S_infty Neumann formula}.
\end{proof}
We next characterize the minimizers of \eqref{S_infty Neumann}.
\begin{proposition}
\label{characterization Neumann minimizer}
Assume that $\bo{G}$ is a compact metric graph without cycles. Let $v_\infty$ be a minimizer of $S_\infty^N(\bo{G})$, and let $x_+, x_- \in \bo{G}$ satisfy
\[
v_\infty(x_+) = 1, \qquad v_\infty(x_-) = -1.
\]
Then
\begin{equation}
\label{maximum distance Neumann}
dist(x_-, x_+) = \max_{x, y \in \bo{G}}dist(x ,y).
\end{equation}
Moreover, if $\bo{U}$ denotes the unique path joining $x_-$ and $x_+$, then
\[
v_\infty(x) = -1+ \frac{2dist(x_-, x)}{dist(x_-,x_+)}, \qquad x \in \bo{U},
\]
and $v_\infty$ is constant on each component of $\bo{G} \setminus \bo{U}$.
\end{proposition}
\begin{proof}
By Proposition~\ref{S_infty Neumann distance} and the proof of the lower bound therein, we have
\begin{align*}
4 &\leq dist(x_-, x_+) \int_{\bo{U}}|\nabla v_\infty|^2 dx\\
&\leq \max_{x, y \in \bo{G}}dist(x, y) \int_{\bo{G}}|\nabla v_\infty|^2 dx =4.
\end{align*}
Hence equality holds throughout. In particular, \eqref{maximum distance Neumann} holds and
\[
\int_{\bo{G} \setminus \bo{U}} |\nabla v_\infty|^2 dx=0.
\]
Therefore, $v_\infty$ is constant on each component of $\bo{G} \setminus \bo{U}$.

Moreover, equality in the Cauchy--Schwarz inequality implies that $\nabla v_\infty$ is constant on $\bo{U}$. Since
\[
v_\infty(x_-) = -1, \qquad v_\infty(x_+) = 1,
\]
we obtain
\[
v_\infty(x) = -1 + \frac{2dist(x_-, x)}{dist(x_-, x_+)}, \qquad x \in \bo{U}.
\]
\end{proof}
To study the asymptotic behavior of $S_p^N(\bo{G})$ as $p \to \infty$, we first establish a uniform upper bound.
\begin{lemma}
\label{uniform bound S_p^N}
There exists a constant $C>0$, independent of $p$, such that
\[
S_p^N(\bo{G}) \leq C
\]
for all sufficiently large $p$.
\end{lemma}
\begin{proof}
Fix an edge $\bo{e} \in \bo{E}$, identified with $[0,\ell(\bo{e})]$, and choose $a \in (0, \ell (\bo{e}))$ and $L > 0$ such that $[a - L, a + L] \subset (0,\ell (\bo{e}))$. Define $\varphi \in H^1(\bo{G})$ by
\[
\varphi(x) =
\begin{cases}
(x - a) \(1 - \dfrac{|x-a|}{L}\),& x \in [a - L, a + L]\subset \bo{e},\\
0, & \text{otherwise}.
\end{cases}
\]
Since
\[
\varphi(a + x) = -\varphi(a - x) \qquad \text{for } x \in [- L, L],
\]
it follows that, for every $p>1$,
\[
\int_{\bo{G}}|\varphi|^{p-1}\varphi\,dx = 0.
\]
Hence $\varphi$ is admissible for the definition of $S_p^N(\bo{G})$. Therefore,
\[
S_p^N(\bo{G}) \leq \frac{\displaystyle\int_{\bo{G}}|\nabla \varphi|^2 dx}{\( \displaystyle\int_{\bo{G}}|\varphi|^{p+1} dx \)^{\frac{2}{p+1}}}.
\]
Since $\varphi$ is fixed and nontrivial,
\[
\( \int_{\bo{G}}|\varphi|^{p+1} dx \)^{\frac{1}{p+1}} \longrightarrow \|\varphi\|_{L^\infty(\bo{G})} > 0 \qquad \text{as } p \to \infty.
\]
Thus the right-hand side is uniformly bounded for all sufficiently large $p$. Hence there exists a constant $C>0$, independent of $p$, such that
\[
S_p^N(\bo{G}) \leq C.
\]
\end{proof}
Using the uniform bound obtained in Lemma~\ref{uniform bound S_p^N}, we now establish the convergence of the optimal Sobolev constants as $p \to \infty$.
\begin{proposition}
\label{S_p^N convergence}
As $p \to \infty$, one has
\[
S_p^N(\bo{G}) \longrightarrow S_\infty^N(\bo{G}).
\]
\end{proposition}
\begin{proof}
We first prove the lower bound. Let $v_p$ be a minimizer of $S_p^N(\bo{G})$ normalized by $\|v_p\|_{L^{p+1}(\bo{G})} = 1$. By Lemma~\ref{uniform bound S_p^N},
\[
\int_{\bo{G}}|\nabla v_p|^2 dx = S_p^N(\bo{G}) \leq C.
\]
Moreover, since $\bo{G}$ has finite length and using H\"{o}lder's inequality,
\[
\|v_p\|_{L^2(\bo{G})} \leq |\bo{G}|^{\frac{1}{2} - \frac{1}{p+1}} \|v_p\|_{L^{p+1}(\bo{G})},
\]
and hence $\{v_p\}$ is bounded in $H^1(\bo{G})$. Thus, up to a subsequence,
\[
v_p \rightharpoonup v_\infty \quad \text{in } H^1(\bo{G}),
\]
and
\[
v_p\to v_\infty \quad\text{uniformly on }\bo{G}.
\]
As in the Dirichlet case, $\|v_\infty\|_{L^\infty(\bo{G})} = 1$. We now use the constraint condition for the Kirchhoff--Neumann problem. By \eqref{Neumann constraint}, we have
\[
\int_{\bo{G}}(v_p^+)^p dx = \int_{\bo{G}}(v_p^-)^p dx,
\]
and hence
\[
\|v_p^+\|_{L^p(\bo{G})} = \|v_p^-\|_{L^p(\bo{G})}.
\]
Since $v_p \to v_\infty$ uniformly on $\bo{G}$, we have
\[
\|v_p^\pm - v_\infty^\pm \|_{L^\infty(\bo{G})} \to 0.
\]
Therefore,
\begin{align*}
\left| \|v_p^\pm\|_{L^p(\bo{G})} - \|v_\infty^\pm \|_{L^p(\bo{G})} \right|
&\leq \|v_p^\pm - v_\infty^\pm \|_{L^p(\bo{G})}\\
&\leq |\bo{G}|^{\frac{1}{p}} \|v_p^\pm - v_\infty^\pm \|_{L^\infty(\bo{G})} \to 0.
\end{align*}
Moreover, $\|v_\infty^\pm \|_{L^p(\bo{G})} \longrightarrow \|v_\infty^\pm \|_{L^\infty(\bo{G})}$. Consequently,
\[
\|v_p^\pm \|_{L^p(\bo{G})} \longrightarrow \|v_\infty^\pm \|_{L^\infty(\bo{G})}.
\]
Passing to the limit in $\|v_p^+ \|_{L^p(\bo{G})} = \|v_p^- \|_{L^p(\bo{G})}$, we obtain
\[
\|v_\infty^+ \|_{L^\infty(\bo{G})} = \|v_\infty^- \|_{L^\infty(\bo{G})}.
\]
Since $\|v_\infty \|_{L^\infty(\bo{G})} = 1$, it follows that
\[
\max_{\bo{G}}v_\infty = 1, \qquad \min_{\bo{G}}v_\infty = -1.
\]
Thus $v_\infty$ is admissible for \eqref{S_infty Neumann}. By the weak lower semicontinuity,
\[
S_\infty^N(\bo{G}) \leq \int_{\bo{G}}|\nabla v_\infty|^2 dx \leq \liminf_{p\to\infty}S_p^N(\bo{G}).
\]

We next prove the upper bound. Let $v_\infty$ be a minimizer of $S_\infty^N(\bo{G})$. For each $ p > 1$, define $F_p:\mathbb{R}\to\mathbb{R}$ by
\[
F_p(a) := \int_{\bo{G}}|v_\infty - a|^{p-1}(v_\infty - a)\,dx, \qquad a \in \mathbb{R}.
\]
Since
\[
\max_{\bo{G}}v_\infty = 1, \qquad \min_{\bo{G}}v_\infty = -1,
\]
the function $F_p$ is continuous and strictly decreasing, and
\[
F_p(-1) > 0, \qquad F_p(1) < 0.
\]
Hence there exists a unique $a_p \in (-1,1)$ such that
\[
F_p(a_p)=0.
\]
Thus $v_\infty - a_p $ is admissible for the definition of $S_p^N(\bo{G})$.

We claim that $a_p \longrightarrow 0$. Indeed, let $\{a_{p_n}\}$ be any convergent subsequence and write $a_{p_n}\to a \in[-1,1]$. Since $F_{p_n}(a_{p_n})=0$, we have
\[
\|(v_\infty - a_{p_n})^+ \|_{L^{p_n}(\bo{G})} = \|(v_\infty - a_{p_n})^- \|_{L^{p_n}(\bo{G})}.
\]
Since $a_{p_n} \to a$, we have
\[
v_\infty - a_{p_n} \longrightarrow v_\infty - a \quad \text{uniformly on } \bo{G}.
\]
By the argument in the proof of the lower bound, we obtain
\[
\|(v_\infty - a)^+ \|_{L^\infty(\bo{G})} = \|(v_\infty - a)^-\|_{L^\infty(\bo{G})}.
\]
Since
\[
\max_{\bo{G}}v_\infty = 1, \qquad \min_{\bo{G}}v_\infty = -1,
\]
this gives
\[
1 - a = 1 + a,
\]
and hence $a=0$. Therefore,
\[
a_p \longrightarrow 0.
\]

Consequently,
\[
v_\infty - a_p \longrightarrow v_\infty \quad\text{uniformly on }\bo{G},
\]
and, as in the Dirichlet case,
\[
\|v_\infty - a_p \|_{L^{p+1}(\bo{G})} \longrightarrow \|v_\infty \|_{L^\infty(\bo{G})} = 1.
\]
Since $\nabla(v_\infty - a_p) = \nabla v_\infty$, we obtain
\begin{align*}
S_p^N(\bo{G}) &\leq \frac{\displaystyle\int_{\bo{G}}|\nabla v_\infty|^2 dx}{\|v_\infty - a_p\|_{L^{p+1}(\bo{G})}^2}.
\end{align*}
Taking the upper limit and using the fact that $v_\infty$ is a minimizer of $S_\infty^N(\bo{G})$, we obtain
\[
\limsup_{p \to \infty}S_p^N(\bo{G}) \leq \int_{\bo{G}}|\nabla v_\infty|^2 dx = S_\infty^N(\bo{G}).
\]
Combining the upper and lower bounds, we conclude that
\[
\lim_{p \to \infty}S_p^N(\bo{G}) = S_\infty^N(\bo{G}).
\]
\end{proof}
We next describe the asymptotic behavior of the normalized minimizers.
\begin{proposition}
\label{convergence Neumann minimizer}
Let $v_p$ be a minimizer of $S_p^N(\bo{G})$ normalized by $\|v_p\|_{L^{p+1}(\bo{G})} = 1$. Then, for every sequence $p_n \to \infty$, there exist a subsequence, still denoted by $p_n$, and a minimizer $v_\infty$ of $S_\infty^N(\bo{G})$ such that
\[
v_{p_n} \longrightarrow v_\infty \quad\text{in } H^1(\bo{G}),
\]
and
\[
v_{p_n}\to v_\infty \quad \text{uniformly on } \bo{G}.
\]
Moreover,
\[
\max_{\bo{G}}v_\infty = 1, \qquad \min_{\bo{G}}v_\infty = -1.
\]
\end{proposition}
\begin{proof}
By Lemma~\ref{uniform bound S_p^N}, the sequence $\{v_{p_n}\}$ is bounded in $H^1(\bo{G})$. Hence, up to a subsequence,
\[
v_{p_n} \weakto v_\infty \quad \text{weakly in }H^1(\bo{G}),
\]
and, by the compact embedding,
\[
v_{p_n}\to v_\infty \quad\text{uniformly on }\bo{G}.
\]
As in the proof of Proposition~\ref{S_p^N convergence}, we have
\[
\max_{\bo{G}}v_\infty = 1, \qquad \min_{\bo{G}}v_\infty = -1.
\]
Thus $v_\infty$ is admissible for \eqref{S_infty Neumann}. By the weak lower semicontinuity and Proposition~\ref{S_p^N convergence},
\begin{align*}
S_\infty^N(\bo{G}) &\leq \int_{\bo{G}}|\nabla v_\infty|^2 dx\\
&\leq \liminf_{n \to \infty} \int_{\bo{G}}|\nabla v_{p_n}|^2 dx\\
&= \lim_{n \to \infty}S_{p_n}^N(\bo{G}) = S_\infty^N(\bo{G}).
\end{align*}
Therefore,
\[
\int_{\bo{G}}|\nabla v_\infty|^2 dx = S_\infty^N(\bo{G}),
\]
and hence $v_\infty$ is a minimizer of $S_\infty^N(\bo{G})$. Moreover,
\[
\int_{\bo{G}}|\nabla v_{p_n}|^2 dx = S_{p_n}^N(\bo{G}) \longrightarrow S_\infty^N(\bo{G}) = \int_{\bo{G}}|\nabla v_\infty|^2 dx.
\]
Since
\[
\nabla v_{p_n}\weakto \nabla v_\infty \quad\text{weakly in }L^2(\bo{G}),
\]
we obtain
\[
\nabla v_{p_n}\longrightarrow \nabla v_\infty \quad\text{in }L^2(\bo{G}).
\]
Together with the uniform convergence, this yields
\[
v_{p_n}\longrightarrow v_\infty \quad\text{in }H^1(\bo{G}).
\]
\end{proof}

\section{Proofs of Theorem~\ref{main3} and Theorem~\ref{main4}} \label{sect;9}
We first prove Theorem~\ref{main3}.
\begin{proof}[Proof of Theorem~\ref{main3}]
Let $\{p_n\}$ be any sequence such that $p_n \to \infty$. By Proposition~\ref{sigma_p^N form}, the function
\[
v_p := \frac{u_p}{\|u_p\|_{L^{p+1}(\bo{G})}}
\]
is a minimizer of $S_p^N(\bo{G})$ satisfying
\[
\|v_p\|_{L^{p+1}(\bo{G})}=1,\qquad u_p = S_p^N(\bo{G})^{\frac{1}{p-1}}v_p.
\] 
By Proposition~\ref{convergence Neumann minimizer}, there exist a subsequence, still denoted by $\{p_n\}$, and a minimizer $v_\infty$ of $S_\infty^N(\bo{G})$ such that
\[
v_{p_n}\longrightarrow v_\infty
\]
strongly in $H^1(\bo{G})$ and uniformly on $\bo{G}$. By Proposition~\ref{S_p^N convergence},
\[
S_{p_n}^N(\bo{G})\longrightarrow S_\infty^N(\bo{G}) > 0.
\]
Hence
\[
S_{p_n}^N(\bo{G})^{\frac{1}{p_n-1}} \longrightarrow 1.
\]
Recalling \eqref{u_p Neumann}, we therefore obtain
\[
u_{p_n}\longrightarrow v_\infty
\]
strongly in $H^1(\bo{G})$ and uniformly on $\bo{G}$. By Proposition~\ref{characterization Neumann minimizer}, there exist
$a_*, b_* \in \bo{G}$ such that
\[
v_\infty(a_*) = -1,\qquad v_\infty(b_*) = 1,
\]
and
\[
(a_*,b_*) \in \mathcal{D}.
\]
Moreover, if $\bo{U}$ denotes the unique path joining $a_*$ and $b_*$, then
\[
v_\infty(x) = -1+\frac{2dist(a_*, x)}{dist(a_*, b_*)}, \qquad x\in\bo{U},
\]
and $v_\infty$ is constant on each connected component of $\bo{G}\setminus \bo{U}$. Thus the statement follows with $u_\infty = v_\infty$.
\end{proof}
We next prove Theorem~\ref{main4}.
\begin{proof}[Proof of Theorem~\ref{main4}]
Suppose, by contradiction, that the conclusion does not hold. Then there exist $\varepsilon > 0$ and a sequence $p_n \to \infty$ such that
\[
\left| dist(x_{p_n}^-, x_{p_n}^+) - \max_{x,y \in \bo{G}}dist(x,y) \right| \geq \varepsilon
\]
for every $n$. By Theorem~\ref{main3}, there exist a subsequence, still denoted by $\{p_n\}$, and a minimizer $u_\infty$ of $S_\infty^N(\bo{G})$ such that
\[
u_{p_n}\longrightarrow u_\infty \quad \text{uniformly on }\bo{G}.
\]
Since $\bo{G}$ is compact, up to a further subsequence,
\[
x_{p_n}^+ \longrightarrow x_+, \qquad x_{p_n}^- \longrightarrow x_-
\]
for some $x_+, x_- \in \bo{G}$. Since $x_{p_n}^+$ and $x_{p_n}^-$ are maximum and minimum points of $u_{p_n}$, respectively, the uniform convergence yields
\[
u_\infty(x_+) = 1, \qquad u_\infty(x_-) = -1.
\]
By Proposition~\ref{characterization Neumann minimizer}, we therefore have
\[
dist(x_-, x_+) = \max_{x, y \in \bo{G}}dist(x,y).
\]
Hence, by the continuity of the distance function,
\[
dist(x_{p_n}^-, x_{p_n}^+) \longrightarrow \max_{x, y \in \bo{G}}dist(x,y).
\]
Therefore,
\[
dist(x_{p_n}^-, x_{p_n}^+) \longrightarrow \max_{x,y \in \bo{G}}dist(x,y),
\]
which contradicts
\[
\left| dist(x_{p_n}^-, x_{p_n}^+) - \max_{x,y \in \bo{G}}dist(x,y) \right| \geq \varepsilon.
\]
Hence,
\[
dist(x_p^-, x_p^+) \longrightarrow \max_{x,y \in \bo{G}}dist(x,y) \qquad\text{as }p \to \infty.
\]
\end{proof}

\appendix
\section{Green function on a $Y$-shaped metric graph}\label{Appendix A}
Let $\bo{G}$ be the $Y$-shaped metric graph considered in the example above. In this appendix, we give the detailed computation of the Dirichlet--Kirchhoff Green function on $\bo{G}$. 

We identify each edge $\bo{e}_i$ with the interval $[0,l_i]$, where $0$ corresponds to the common interior vertex $\bo{v}_0$ and $l_i$ corresponds to the end vertex $\bo{v}_i$, for $i=1,2,3$. Let $a\in\bo{e}_1$, and let
\[
s=dist(a,\bo{v}_0)\in(0,l_1).
\]
We compute the Dirichlet--Kirchhoff Green function $G(\cdot,a)$. For simplicity, set
\[
g(x)=G(x,a), \qquad C_1=G(\bo{v}_0,a), \qquad C_2 = G(a,a).
\]
Since $-\Delta g = \delta_a$, the function $g$ is linear on each edge segment away from the pole $a$. By the continuity of $g$ and the Dirichlet conditions at the end vertices, the restriction of $g$ to $\bo{e}_1$ is given by
\[
g_1(t) =
\begin{cases}
C_1+\dfrac{C_2 - C_1}{s}t, & 0 \leq t \leq s,\\
C_2\dfrac{l_1-t}{l_1-s}, & s \leq t \leq l_1,
\end{cases}
\]
while its restrictions to $\bo{e}_2$ and $\bo{e}_3$ are given by
\[
g_i(t) = C_1\(1-\frac{t}{l_i}\),\qquad (i = 2,3).
\]

We first derive the Kirchhoff condition at $\bo{v}_0$. Since the coordinate on each edge increases away from $\bo{v}_0$, the outgoing derivatives at $\bo{v}_0$ are
\[
\frac{C_2 - C_1}{s}, \qquad -\frac{C_1}{l_2}, \qquad -\frac{C_1}{l_3}.
\]
Therefore, the Kirchhoff condition at $\bo{v}_0$ gives
\begin{equation}
\label{system1}
\frac{C_2 - C_1}{s} -\frac{C_1}{l_2} -\frac{C_1}{l_3} =0.
\end{equation}

Next, we derive the jump condition at the pole $a$. On $\bo{e}_1$, the Green function satisfies
\[
-\Delta_t G(t,a) = \delta_s.
\]
Integrating this equation over $(s - \varepsilon, s + \varepsilon)$, we obtain
\[
-\int_{s - \varepsilon}^{s + \varepsilon}\Delta_t G(t,a) dt =1.
\]
Hence,
\[
\nabla_t G(s-0,a) - \nabla_t G(s+0,a) = 1.
\]
From the explicit affine expressions above,
\[
\nabla_t G(s-0,a) = \frac{C_2 - C_1}{s}, \qquad \nabla_t G(s+0,a) = -\frac{C_2}{l_1 - s}.
\]
Thus, the jump condition becomes
\begin{equation}
\label{system2}
\frac{C_2 - C_1}{s} +\frac{C_2}{l_1-s} = 1.
\end{equation}

Solving the two equations \eqref{system1} and \eqref{system2}, we obtain
\[
G(a,a) = \frac{(l_1 - s) \left( s + \dfrac{l_2l_3}{l_2+l_3} \right)}{ l_1 + \dfrac{l_2l_3}{l_2 + l_3} }.
\]
The corresponding formulas on $\bo{e}_2$ and $\bo{e}_3$ follow by permuting the indices.

\section*{Acknowledgements}
The author would like to express his sincere gratitude to Professor Futoshi Takahashi for his valuable guidance and helpful discussions throughout this work. The author would also like to thank Tatsuya Hosono for valuable discussions and helpful comments.


\begin{thebibliography}{99}
\bibitem{A-G}
Adimurthi and M. Grossi;
{\it Asymptotic estimates for a two-dimensional problem with polynomial nonlinearity},
\newblock Proc. Amer. Math. Soc. 132 (2004), no. 4, 1013-1019.


\bibitem{Kurata-Shibata}
K. Kurata and M. Shibata;
{\it Least energy solutions to semi-linear elliptic problems on metric graphs},
\newblock J. Math. Anal. Appl. 491 (2020), 124297.


\bibitem{M-M-P}
F. De Marchis, L. Mazzuoli, F. Pacella;
{\it Stability and asymptotic behavior of one-dimensional solutions in cylinders},
\newblock J. Differential Equations 462 (2026), 114146.


\bibitem{P-W}
E. Parini and T. Weth;
{\it Existence, unique continuation and symmetry of least energy nodal solutions to sublinear Neumann problems},
\newblock Math. Z. 280 (2015), 707-732.


\bibitem{R-W1994}
X. Ren and J. Wei; 
{\it On a two dimensional elliptic problem with large exponent in nonlinearity},
\newblock Trans. Amer. Math. Soc. 343 (1994), 749–763.


\bibitem{R-W1996}
X. Ren and J. Wei;
{\it Single-point condensation and least-energy solutions},
\newblock Proc. Amer. Math. Soc. 124 (1996), 111–120.


\bibitem{S-T2018}
A. Salda\~{n}a and H. Tavares;
{\it Least energy solutions of Hamiltonian elliptic systems with Neumann boundary conditions},
\newblock J. Differential Equations 265 (2018), 6127-6165.


\bibitem{S-T2022}
A. Salda\~{n}a and H. Tavares;
{\it On the least-energy solutions of the pure Neumann Lane--Emden equation},
\newblock Nonlinear Differential Equations Appl. 29 (2022), Article No. 30.


\bibitem{Takahashi2014}
F. Takahashi;
{\it Asymptotic behavior of least energy solutions for a 2D nonlinear Neumann problem with large exponent},
\newblock J. Math. Anal. Appl. 411 (2014), 95-106.


\bibitem{Takahashi}
F. Takahashi;
{\it Notes on asymptotic behavior of radial solutions for some weighted elliptic equations on the annulus},
\newblock Partial Differ. Equ. Appl. 5 (2024), no. 4, Paper No. 25, 15 pp.

\end{thebibliography}
\end{document}